\documentclass[preprint,12pt]{elsarticle}
\usepackage{graphicx}
\usepackage{tikz}
\usetikzlibrary{positioning, arrows.meta}
\usepackage{amssymb}
\usepackage{algcompatible}
\usepackage{algpseudocode}
\usepackage{multirow}
\usepackage{multirow}
\usepackage{xspace}
\usepackage{hyperref}
\usepackage{listings}
\usepackage{url}
\usepackage{csquotes}
\usepackage[mathscr]{euscript} 
\usepackage{amsfonts,epsfig}
\usepackage{enumerate}
\usepackage{mathtools,halloweenmath}
\usepackage{amsmath,amscd,amsfonts,amssymb}

\usepackage{lineno,hyperref}
\usepackage{graphicx,subfigure}
\usepackage{amsfonts,epsfig}
\newcommand\rnumber{\operatorname{-}}
\usepackage[mathscr]{euscript} 
\usepackage{enumerate}
\usepackage{amsmath,amscd,amsfonts,amssymb}
\usepackage{times}
\usepackage{extarrows}
\usepackage{stmaryrd}
\usepackage{textcomp}
\usepackage[mathscr]{euscript}
\usepackage{algorithm2e}
\usepackage{colortbl}
\usepackage{color,soul}
\usepackage{newunicodechar}
\usepackage{amsthm}
\theoremstyle{definition}
\newtheorem{thm}{Theorem}
\newtheorem{cor}{Corollary}
\newtheorem{lem}{Lemma}
\newtheorem{rem}{Remark}
\newtheorem{defn}{Definition}
\newtheorem{exa}{Example}

\newcommand{\cX}{\underline{\bf X}}

\newcommand{\cY}{\underline{\bf Y}}

\newcommand{\cZ}{\underline{\bf Z}}
\newcommand{\A}{{\bf A}}
\newcommand{\x}{{\bf x}}

\newcommand{\X}{{\bf X}}
\newcommand{\Y}{{\bf Y}}

\newcommand{\B}{{\bf B}}
\newcommand{\C}{{\bf C}}
\newcommand{\Q}{{\bf Q}}

\newcommand{\U}{{\bf U}}

\newcommand{\V}{{\bf V}}

\usepackage{xcolor}

\usepackage{amssymb}

\begin{document}

\begin{frontmatter}

{\title{Fast randomized Kronecker tensor decomposition: algorithms and error analysis}}

\author[label1,label1n]{Salman Ahmadi-Asl}
\author[label2]{Naeim Rezaeian}
\author[label3]{André L. F. de Almeida}
\author[label4]{Yipeng Liu}

\affiliation[label1]{organization={Research Center of the Artificial Intelligence Institute, Innopolis University, 420500 Innopolis, Russia},
          }
          
\affiliation[label1n]{organization={Lab of Machine Learning and Knowledge Representation, Innopolis University, 420500 Innopolis, Russia},
          }
\affiliation[label2]{organization={Peoples' Friendship University of Russia, Moscow, Russia},
}
\affiliation[label3]{organization={Department of Teleinformatics Engineering, Federal University of Ceara, Fortaleza, Brazil},
           }    
\affiliation[label4]{organization={School of Information and Communication Engineering, University of Electronic Science
and Technology of China (UESTC), Chengdu, 611731, China},
           }

\begin{abstract}
This paper proposes fast randomized algorithms for computing the Kronecker Tensor Decomposition (KTD) by replacing the sequence of deterministic SVDs in the TTr1SVD framework with randomized SVDs incorporating oversampling and power iterations. The proposed algorithms can decompose a given tensor into the KTD format significantly faster than existing state-of-the-art deterministic methods. Our principal idea is to use randomization to reduce computational complexity while maintaining controlled accuracy. A detailed theoretical analysis is presented, including a recursive error bound that accounts for error propagation through the TTr1SVD tree structure. We prove that the expected Frobenius norm error is bounded by a sum of tail energies multiplied by factors that decay exponentially with the number of power iterations. Extensive simulations on synthetic and real-world datasets demonstrate several orders of magnitude acceleration compared to the deterministic approach, with applications to tensor completion, video/image compression, image denoising, and image super-resolution.
\end{abstract}


\end{frontmatter}

\section{Introduction}  
Tensors, also known as multi-way arrays, generalize matrices to higher dimensions. They can be considered data structures for storing and manipulating multidimensional data. Tensors have applications in various fields, including mathematics, physics, computer science, and engineering. In computer science, tensors are widely used in machine learning and deep learning algorithms. They represent and manipulate multi-dimensional data such as images, videos, and text. Tensors are also used in natural language processing to represent and process words and sentences, capturing their sequential and hierarchical structure. The interested reader is referred to \cite{comon2009tensor,FavierAlmeida2014,cichocki2016tensor,asante2021matrix,ahmadi2023fast,Sidiropoulos2017,elden2019solving} and the references therein for more details about tensors and their applications.

Tensor decompositions allow the breaking down of a higher-order tensor into a set of lower-order tensors that capture the underlying structure and patterns in the data. 
Unlike the matrix case, where the concept of rank is unique, there are several types of tensor decompositions due to the lack of a unique concept of rank for tensors. Such decompositions include Canonical Polyadic decomposition (CPD) \cite{hitchcock1927expression}, HSOVD \cite{de2000multilinear}, Tensor Train (TT) decomposition \cite{oseledets2011tensor,Yassin_Andre}, Tensor Ring decomposition \cite{zhao2016tensor}, block term decomposition \cite{de2008decompositionsI}, constrained factor tensor decompositions \cite{Almeida_Elsevier_2007,confac,stegeman2009},
and Kronecker Tensor Decomposition (KTD) \cite{phan2012revealing,batselier2017constructive}. 
 
The KTD has proven to be a versatile tool, with applications in data completion~\cite{phan2013tensor}, feature extraction~\cite{phan2013basis}, data compression~\cite{batselier2017constructive}, hypergraph analysis~\cite{pickard2023kronecker}, large-language model compression~\cite{tahaei2022kroneckerbert,edalati2022krona} and the development of lightweight recurrent neural networks~\cite{wang2021kronecker}. For example, in \cite{batselier2017constructive}, it is empirically shown that the KTD can provide a much better compression ratio than the HOSVD. However, existing algorithms for computing the KTD are not scalable and struggle with large-scale data tensors. Motivated by the interesting applications of the KTD and the lack of scalable algorithms for it, we develop fast, randomized algorithms to decompose a tensor into the KTD format. 

To the best of our knowledge, this is the first work that not only proposes a fast randomized algorithm for KTD but also provides a rigorous error propagation analysis through the recursive TTr1SVD tree structure, establishing explicit bounds that quantify the trade-off between computational efficiency and approximation accuracy. Indeed, randomized algorithms for tensor decomposition offer several advantages over traditional deterministic methods. They are often faster and more memory-efficient, making them well-suited for large-scale datasets. Additionally, randomized algorithms can provide approximate tensor decompositions with controlled accuracy, allowing for trade-offs between computational cost and solution quality. For randomized approaches to various tensor decompositions, see \cite{ahmadi2020randomized,ahmadi2021randomized,che2019randomized,minster2020randomized,zhang2018randomized}. Randomized and sampling-based methods have recently gained traction for scalable low-rank tensor decompositions. Zhang et al.~\cite{zhang2023randomized} proposed randomized sampling to accelerate low-tubal-rank plus sparse tensor recovery, avoiding full tensor SVD computations. Qin et al.~\cite{qin2024nonconvex} embedded randomized SVD into a nonconvex framework for robust high-order tensor completion, reducing time and memory costs. Larsen and Kolda~\cite{larsen2022practical} introduced leverage-based sampling for the CPD, selecting informative subtensors via approximate leverage scores. Malik and Becker~\cite{malik2018low} developed TensorSketch, a count-sketch method for Tucker decomposition with near-linear time complexity. For tensor ring decomposition, Malik and Becker~\cite{malik2021sampling} proposed a sampling-based method that selectively samples fibers, avoiding large intermediate matrices and scaling to high-order tensors. Together, these works show that randomized techniques—sketching, leverage sampling, and fiber sampling—enable scalable tensor decompositions across multiple formats (Tucker, tensor ring, and tubal rank). Also, see \cite{ahmadi2024adaptive} and \cite{ahmadi2024robust} for sampling-based methods for the T-SVD model.

It is worth highlighting that randomized algorithms are applicable to computing low-rank approximations of data tensors with low-rank structures, such as images and videos.

We can summarize our main contributions as follows:

 \begin{itemize}
     \item {Proposing the first fast randomized algorithms for computing the KTD, replacing all deterministic SVDs in the TTr1SVD framework with randomized SVDs.}
     
     \item {Providing a rigorous theoretical error analysis of the proposed randomized KTD algorithm, including a recursive error bound that accounts for propagation through the TTr1SVD tree structure and demonstrating exponential decay of the error with the number of power iterations.}
     
     \item {Presenting a detailed complexity analysis showing reduction from $\mathcal{O}(I^{N+1})$ for deterministic KTD to $\mathcal{O}(I^N R)$ for our randomized approach, with explicit comparison to traditional methods.}
     
     \item {Validating the proposed algorithms through extensive simulations on synthetic and real-world datasets, demonstrating several orders of magnitude acceleration while maintaining competitive accuracy.}
 \end{itemize}

The outline of this paper is as follows. The preliminaries are presented in Section \ref{Sec:pre}. In Section \ref{Sec:KTD}, we introduce the KTD and its main properties. The randomization method for low-rank approximation of matrices is discussed in Section \ref{Sec:rand_matrix}. Section \ref{RKTD} is devoted to presenting the proposed randomized algorithms for computing the KTD. The simulations are outlined in Section \ref{Sec:Sim}, and finally, the conclusions follow in Section \ref{Sec:Con}.

\section{Preliminaries}\label{Sec:pre}
In this study, we utilize the same notations as in \cite{cichocki2016tensor}. Thus, we use an underlined bold capital letter, a bold capital letter, and a bold lowercase letter to represent a tensor, a matrix, and a vector. $\|.\|$ represents the Frobenius norm of matrices or tensors. The spectral norm of matrices and the Euclidean norm of vectors are represented by the notation $\|.\|_2$. $\mathbb{E}$ represents the mathematical expectation. The following definitions are now provided, which are necessary for our presentation. 

A tensor can be unfolded or metricized along its different modes. For a given tensor $\cX\in\mathbb{R}^{I_1\times I_2\times \cdots\times I_N}$, the $n$-mode unfolding is denoted by $\X_{(n)}\in\mathbb{R}^{I_n\times I_1\ldots i_{n-1}I_{n+1}\ldots I_N}$ and is obtained by stacking the $n$-mode fibers\footnote{This means we fix all modes except mode $n$.} In MATLAB, this can be computed through reshaping and permutation as follows:
\begin{eqnarray*}
 \cY &\leftarrow& {\rm permute}(\cX,[n,1,\ldots,n-1,n+1,\ldots,N]),\\
 \X_{(n)} &\leftarrow& {\rm reshape}(\cY,[I_n,I_1\ldots I_{n-1}I_{n+1}\ldots I_N]).
\end{eqnarray*}
A matrix is transformed into a vector via a so-called vectorization operation. The vectorization of a matrix is performed by stacking its columns one by one. The vectorization operation for tensors is the process of reordering the elements of a tensor to a vector. The notation ``vec'' denotes the vectorization operation of matrices and tensors. The vectorization process of a tensor $\cX$ can be computed based on the matrix vectorization via ${\rm vec}(\cX)={\rm vec}(\X_{(1)})$. 

The Kronecker product is denoted by the symbol $\otimes$, and is defined as 
$$\C=\A\otimes \B=\begin{bmatrix}
 a_{11}\B & \cdots &  a_{1n}\B\\
 \vdots & \dots & \vdots\\
a_{m1}\B & \dots & a_{mn}\B 
\end{bmatrix}\in\mathbb{R}^{mp\times nq},$$ 
where $\A\in\mathbb{R}^{m\times n}$ and $\B\in\mathbb{R}^{p\times q}$ are two matrices, and $a_{ij}$ is an element of the matrix $\A$. 

\begin{defn}\label{def:outer}
(Outer product of tensors) The outer product of tensors $\cX\in\mathbb{R}^{I_1\times I_2\times\cdots\times I_N}$ and $\cY\in\mathbb{R}^{J_1\times J_2\times \cdots\times J_M}$ is denoted by $\cZ=\cX\circ\cY$ an ($N+M$)-th order tensor of size $I_1\times\cdots\times I_N\times J_1\times\cdots\times J_M$. The elements of the tensor $\cZ$ can be presented as  $z_{i_1,\cdots,i_N,j_1,\ldots,j,_M}=x_{i_1,\cdots,i_N}y_{j_1,\cdots,j_M}$.     
\end{defn}
From definition \ref{def:outer}, we see that the outer product of two and three vectors gives a matrix and a third-order tensor, respectively. It can be easily shown that 
\begin{eqnarray}\label{vecm}
{\rm vec}({\bf a}\circ{\bf b})={\bf b}\otimes{\bf a},
\end{eqnarray}
and in general, we have 
\begin{eqnarray}
 {\rm vec}(\underline{\bf A}\circ\underline{\bf B})={\rm vec}(\underline{\bf B})\otimes{\rm vec}(\underline{\bf A}).   
\end{eqnarray}
The tensor Kronecker product is a generalization of the matrix Kronecker product introduced in \cite{phan2012revealing,batselier2017constructive}. We use the same notation for the tensor Kronecker product. The tensor Kronecker  product is defined as follows
\begin{defn}
(Tensor Kronecker product) \cite{phan2012revealing}. Let $\cX\in\mathbb{R}^{J_1\times J_2\times \cdots \times J_N}$ and $\cY\in\mathbb{R}^{K_1\times K_2\times \cdots\times K_N}$. The Kronecker tensor product of $\cX$ and $\cY$ is an $N$-th order tensor expressed as, $\cZ=\cX\otimes \cY\in\mathbb{R}^{I_1\times I_2\times \cdots\times I_N},\,I_n=J_nK_n$ and its elements are defined as $z_i=x_jy_k$ where $i=[i_1,i_2,\ldots,i_N],\,j=[j_1,j_2,\ldots,j_N],\,k=[k_1,k_2,\ldots,k_N]$ and $i_n=k_n+(j_n-1)K_n$.    
\end{defn}
It is not difficult to check that partitioning $\underline{\bf Z}$ to a block tensor of size $J_1\times J_2\times \cdots \times J_N$, each block-$j$ ($j=[j_1,j_2,\ldots,j_N]$) can be represented as $x_j\cY$ \cite{phan2012revealing}. In \cite{batselier2015constructive}, the authors use an index-merging strategy to extend the matrix Kronecker product to the tensor case. Starting with the matrix Kronecker product, for two matrices $\A\in\mathbb{R}^{I\times J}$ and $\B\in\mathbb{R}^{K\times L}$, the Kronecker product $\C=\A\otimes\B\in\mathbb{R}^{IK\times JL}$, can be represented element-wise as $\C_{[i_1i_3][i_2i_4]}=\A_{i_3i_4}\B_{i_1i_2}$ where $[i_1i_3]=i_1+K(i_3-1)$ and $[i_2i_4]=i_2+L(i_4-1)$. Using this idea, the tensor Kronecker product of two tensors $\cX\in\mathbb{R}^{J_1\times J_2\times\cdots\times J_N}$ and $\cY\in\mathbb{R}^{K_1\times K_2\times\cdots\times K_N}$, that is $\cZ=\cX\otimes\cY\in\mathbb{R}^{J_1K_1\times J_2K_2\times J_NK_N}$ can be represented element-wise as 
\begin{eqnarray}
\cZ_{[i_1i_{N+1}][i_2i_{N+2}]\cdots[i_Ni_{2N}]}=\cX_{i_{N+1}i_{N+2}\cdots i_{2N}}\cY_{i_1i_2\cdots i_N},    
\end{eqnarray}
where $[i_ni_{N+n}]=i_n+(i_{N+n}-1)K_n,\,n=1,2,\ldots,N$. The Kronecker product of two tensors can be computed in MATLAB using vectorization and reshaping/permutation. More precisely, let $\cX\in\mathbb{R}^{I_1\times I_2\times \cdots\times I_N}$ and $\cY\in\mathbb{R}^{J_1\times J_2\times \cdots\times J_N}$ be given; then we need to first vectorize them as $\cX(:)$ and $\cY(:)$ and then compute their classical Kronecker product as ${\bf c}=\cX(:)\otimes\cY(:)$. Next, the long vector ${\bf c}$ is reshaped to a tensor of size $J_1\times \cdots\times J_N\times I_1\times\cdots\times I_N$ and then permute the resulting tensor with the permutation vector ${\bf p}$ as
\begin{eqnarray}\label{permu}
 {\bf p}=[1,N+1,2,N+2,3,N+3\ldots,N,2N].
\end{eqnarray}
Finally, the permuted tensor is reshaped to an $N$-order tensor of size $I_1J_1\times I_2J_2\times \cdots\times I_NJ_N$.
The MATLAB code for this operation and related useful functions is available in the GitHub repository {https://github.com/kbatseli/TKPSVD}.

The following lemma can be proved using the fact \eqref{vecm} and will be used in our analyses.
\begin{lem}\label{lm1}
Let ${\bf a},\,{\bf b}$ and ${\bf c}$ be three vectors; then ${\rm vec}({\bf a}\circ{\bf b}\circ{\bf c})={\bf c}\otimes{\bf b}\otimes{\bf a}$.   
\end{lem}
From Lemma \ref{lm1}, we can see
\begin{eqnarray}\label{rela}
{\rm vec}(\cX)={\rm vec}({\bf X}_{(1)})={\bf c}\otimes{\bf b}\otimes{\bf a}={\rm vec}({\bf a}\circ({\bf c}\otimes{\bf b})),    
\end{eqnarray}
which will be used in Section \ref{Sec:KTD}.

\section{Kronecker tensor decomposition}\label{Sec:KTD}
The KTD is a technique that expresses a higher-order tensor as a Kronecker product of smaller factor tensors. First introduced by Phan et al. \cite{phan2012revealing}, KTD extends Kronecker matrix decomposition (KMD) \cite{van1993approximation} to higher-order tensors. It is particularly useful for simplifying tensor representations, reducing computational complexity, and supporting applications in structural pattern extraction \cite{phan2013basis}, and tensor completion \cite{phan2013tensor}.

Beyond its original formulation, the KTD has been adapted to diverse fields. For instance, KMD has been applied to large-language model compression \cite{tahaei2022kroneckerbert,edalati2022krona}, while Batselier \cite{batselier2017constructive} developed a constructive algorithm for computing the KTD using a CPD of a reshaped tensor with orthogonal rank-1 terms, based on the TTr1SVD \cite{batselier2015constructive,salmi2009sequential}. This approach also enabled rigorous error analysis and demonstrated the method's effectiveness in applications such as image compression.

The effectiveness of KMD has been substantiated in several studies. For instance, \cite{cai2022kopa} compared low-rank approximations of a standard \(512 \times 512\) cameraman image. Their results showed that while the optimal rank-1 approximation using Singular Value Decomposition (SVD) captured only \(45.63\%\) of the image's total variation, a Kronecker-based rank-1 approximation with a \((16 \times 32) \otimes (32 \times 16)\) patch structure captured \(77.55\%\)—a substantial gain achieved without increasing model complexity, as both approximations used the same number of parameters (1023). A visual comparison is provided in Fig.~\ref{fig:placeholder}.

\begin{figure}[ht]
    \centering
    \includegraphics[width=0.6\linewidth]{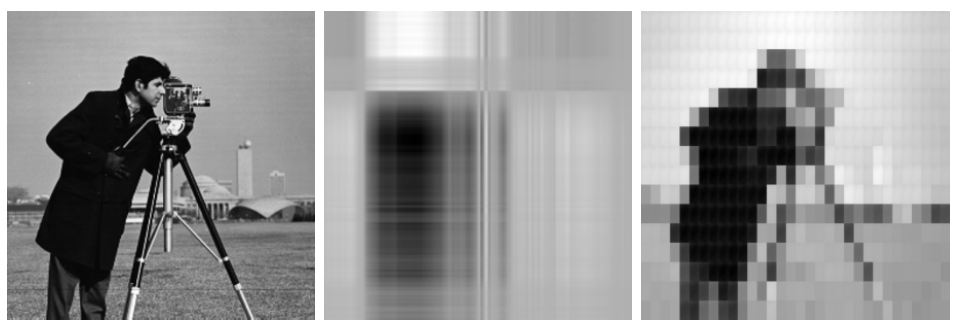}
    \caption{(Left) Original cameraman image; (Middle) SVD rank-1 approximation; (Right) KPD rank-1 approximation. Reproduced from \cite{cai2022kopa}.}
    \label{fig:placeholder}
\end{figure}

Inspired by the success of KMD, its tensor counterpart—(KTD)—has been introduced and applied across domains such as signal processing, image analysis, and machine learning. The KTD facilitates meaningful feature extraction and dimensionality reduction in high-dimensional data. For example, \cite{wang2021kronecker} applied the KTD to compress recurrent neural networks, demonstrating its utility in model compression. Further applications continue to underscore the versatility of the decomposition.

The KTD of an $N$-th order tensor $\cX\in\mathbb{R}^{I_1\times I_2\times \cdots\times  I_N}$ admits the model
\begin{eqnarray}
\cX\approx \sum_{r=1}^R \sigma_r\, \cX^{(1)}_r \otimes \cX^{(2)}_r \otimes \cdots \otimes \cX^{(M)}_r,
\end{eqnarray}
where $\cX^{(n)}_r\in\mathbb{R}^{J^{(m)}_1\times J^{(m)}_2\times \cdots\times J^{(m)}_N}$ and 
\begin{eqnarray}
||\cX^{(m)}_r||_F=1,\quad \prod_{m=1}^M {J^{(m)}_n}=I_n,\quad n=1,2,\ldots,N.    
\end{eqnarray}
The normalization $\|\cX^{(m)}_r\|_F = 1$ removes scaling ambiguity inherent in the Kronecker product representation, as any scaling factor can be absorbed into the singular value $\sigma_r$. This is standard in KTD literature and ensures the uniqueness of the decomposition up to permutation.
The minimum number of terms representing a tensor in the KTD format is called the KTD rank. Computing the KTD rank is NP-hard, similar to the CPD rank, as there is a close relation between the KTD and CPD shown in \cite{phan2012revealing,batselier2017constructive}. Due to the special index-merging structure of the KT product, one can reformulate the computation of the KTD as the CPD of a new tensor via reshaping and permutation. To be more precise, let $\cX\in\mathbb{R}^{J_1\times J_2\times J_3}$ and $\cY\in\mathbb{R}^{K_1\times K_2\times K_3}$, and $\cZ=\cX\otimes\cY\in\mathbb{R}^{J_1K_1\times J_2K_2\times J_3K_3}$ for which we have $\cZ_{[i_1i_4][i_2i_5][i_3i_6]}=\cX_{i_4i_5i_6}\cY_{i_1i_2i_3}$. By a sequence of reshaping and permutation described in the following
\begin{eqnarray*}
\cZ_{[i_1i4][i_2i_5][i_3i_6]}\xrightarrow{\rm Reshaping}\cZ_{i_1,i_4,i_2,i_5,i_3,i_6}\xrightarrow{\rm Permuting}\\
\cZ_{i_1,i_2,i_3,i_4,i_5i_6}\xrightarrow{\rm Reshaping}\cZ_{[i_1i_2i_3][i_4i_5i_6]},
\end{eqnarray*}
we can easily see that the KT product of two three-order tensors after these processes converted to the outer product of the vectorizations of the tensor $\cY$ and $\cX$. So, the KTD of the tensor $\cZ$ can be represented as a rank-1 matrix approximation, which can be computed via the SVD, a special case of CPD. Similarly, we can always convert the KTD computation to the CPD after the reshaping and permutation described above. The next theorem illustrates this fact.

\begin{thm}\label{th1} \cite{batselier2017constructive} Given an $N$-th order tensor $\cX\in\mathbb{R}^{I_1\times I_2\times \cdots\times I_N}$, if 
\begin{eqnarray}\label{eq_ktd}
\cX&\approx& \sum_{r=1}^R \sigma_r\, \cX^{(1)}_r \otimes \cX^{(2)}_r \otimes \cdots \otimes \cX^{(M)}_r,\\\nonumber
{\rm and}\\\label{eq_ktd2}
\widetilde{\cX} &\approx& \sum_{r=1}^R \sigma_r\, \x_r^{(M)}\circ\x_r^{(M-1)}\circ\cdots\circ\x_r^{(1)},
\end{eqnarray}
where $\widetilde{\cX}$ is the permutation of $\cX$ such that the indices of the $\x_r^{(n)}$ vectors are identical to those of the $N$-th order tensor $\cX^{(i)}_r$ tensors, then $\x_r^{(n)}={\rm vec}(\cX^{(n)}_r)$ for all $r=1,2,\ldots,R$ and $n=1,2,\ldots,N.$
\end{thm}

We now describe the computational procedure for the KTD. According to Theorem~1, computing the KTD involves three main steps: (1) reshaping and permuting the original tensor, (2) computing a canonical polyadic decomposition (CPD) of the rearranged tensor, and (3) reshaping the resulting factor vectors into the core tensors of the KTD. The central step is therefore the computation of a CPD of the form in \eqref{eq_ktd2}.

Several established algorithms can compute such a CPD, including CP-ALS~\cite{harshman1970foundations,carroll1970analysis}, methods incorporating line search and extrapolation~\cite{rajih2008enhanced,tomasi2006practical,chen2011new}, and compression-based approaches~\cite{kiers1998three}. These methods do not inherently enforce orthogonality among the rank-1 terms in \eqref{eq_ktd2}. However, for tractable error analysis, orthogonality is required, as discussed in \cite{batselier2015constructive}. A CPD with orthogonal rank-1 terms is known as a TTr1 decomposition, and an efficient algorithm, TTr1SVD, has been proposed for its computation. Its MATLAB implementation is available at \url{https://github.com/kbatseli/TKPSVD}.

We now briefly outline the TTr1SVD algorithm. For clarity, consider a third-order tensor $\underline{\bf X}\in\mathbb{R}^{I_1\times I_2\times I_3}$, and let $\mathbf{X}_{(1)}\in\mathbb{R}^{I_1\times I_2 I_3}$ denote its mode-1 unfolding. The truncated SVD of $\mathbf{X}_{(1)}$ is given by
\begin{equation}\label{eq:svd1}
\mathbf{X}_{(1)}\approx\sigma_1 \mathbf{u}_1\circ\mathbf{v}_1 + \sigma_2 \mathbf{u}_2\circ\mathbf{v}_2 + \cdots + \sigma_R \mathbf{u}_R\circ\mathbf{v}_R,
\end{equation}
where $\mathbf{u}_i\in\mathbb{R}^{I_1}$ and $\mathbf{v}_i\in\mathbb{R}^{I_2 I_3}$. Next, each right singular vector $\mathbf{v}_i$ is reshaped into a matrix $\mathbf{V}_i = \operatorname{reshape}(\mathbf{v}_i, [I_2, I_3])$, and its truncated SVD is computed:
\begin{equation}\label{eq:svd2}
\mathbf{V}_i \approx \sigma_{i1} \mathbf{u}_{i1}\circ\mathbf{v}_{i1} + \sigma_{i2} \mathbf{u}_{i2}\circ\mathbf{v}_{i2} + \cdots + \sigma_{iR'} \mathbf{u}_{iR'}\circ\mathbf{v}_{iR'}, \quad i=1,\dots,R.
\end{equation}
Vectorizing both sides of \eqref{eq:svd2} and substituting into \eqref{eq:svd1} yields
\begin{equation}
\mathbf{X}_{(1)}\approx\sum_{r=1}^{R}\sum_{r'=1}^{R'}\sigma_r\sigma_{r'r}\,\mathbf{u}_r \circ (\mathbf{v}_{r'r}\otimes \mathbf{u}_{r'r}).
\end{equation}
Finally, reshaping $\mathbf{X}_{(1)}$ back to a tensor of size $I_1\times I_2\times I_3$ gives the TTr1 decomposition
\begin{equation}\label{eq:out}
\underline{\bf X}\approx\sum_{r=1}^{R}\sum_{r'=1}^{R'}\sigma_r\sigma_{r'r}\,\mathbf{u}_r \circ \mathbf{u}_{r'r} \circ \mathbf{v}_{r'r},
\end{equation}
where we have used the identity \eqref{rela}.

This recursive process generates a tree structure, illustrated in Figure ~\ref{fig:ttr1_tree} for $R=3$ and $R'=2$. The first SVD (level 0) is applied to $\mathbf{X}_{(1)}$. At level 1, SVD is applied to each matrix $\mathbf{V}_i$ obtained from level 0. Subsequent levels are defined analogously.

\begin{figure}[htbp]
    \centering
    \begin{tikzpicture}[
        level distance=1.5cm,
        level 1/.style={sibling distance=4cm},
        level 2/.style={sibling distance=2cm},
        level 3/.style={sibling distance=1cm},
        every node/.style={draw, circle, minimum size=0.7cm, inner sep=1pt},
        edge from parent/.style={draw, -Stealth, thick},
        leaf/.style={draw=none, rectangle, minimum size=0cm, inner sep=2pt}
    ]
    
    \node (root) {$0$}
        child { node {$\sigma_{1}$}
            child { node {$\sigma_{11}$}
                child { node[leaf] {$\tilde{\sigma}_1$} }
                child { node[leaf] {$\tilde{\sigma}_2$} }
            }
            child { node {$\sigma_{12}$}
                child { node[leaf] {$\tilde{\sigma}_3$} }
                child { node[leaf] {$\tilde{\sigma}_4$} }
            }
        }
        child { node {$\sigma_{2}$}
            child { node {$\sigma_{21}$}
                child { node[leaf] {$\tilde{\sigma}_5$} }
                child { node[leaf] {$\tilde{\sigma}_6$} }
            }
            child { node {$\sigma_{22}$}
                child { node[leaf] {$\tilde{\sigma}_7$} }
                child { node[leaf] {$\tilde{\sigma}_8$} }
            }
        }
        child { node {$\sigma_{3}$}
            child { node {$\sigma_{31}$}
                child { node[leaf] {$\tilde{\sigma}_9$} }
                child { node[leaf] {$\tilde{\sigma}_{10}$} }
            }
            child { node {$\sigma_{32}$}
                child { node[leaf] {$\tilde{\sigma}_{11}$} }
                child { node[leaf] {$\tilde{\sigma}_{12}$} }
            }
        };

    \end{tikzpicture}
    \caption{The visualization of the TTr1 decomposition tree representation. Here, each $\tilde{\sigma}_i$ is the product of all nodes down a branch, for example $\tilde{\sigma}_{1}={\sigma}_{1}\sigma_{11}$.}
    \label{fig:ttr1_tree}
\end{figure}

Observe that \eqref{eq:out} is a CPD of the tensor $\underline{\bf X}$ with orthogonal rank-1 terms; a proof of orthogonality is provided in \cite{batselier2017constructive}. A key feature of this formulation is that the singular values $\sigma_r \sigma_{r'r}$ in \eqref{eq:out} behave analogously to singular values in matrix SVD, allowing truncation of the summation when these values become relatively small. 

If a tensor $\underline{\bf X}$ admits an exact representation \eqref{eq_ktd}, then the relative approximation error of the orthogonal CPD with truncation rank $R'$ is given by
\begin{equation}\label{rela:for}
\frac{\|\widetilde{\underline{\bf X}} - \sum_{r=1}^{R'} \sigma_r\, \mathbf{x}_r^{(1)} \circ \mathbf{x}_r^{(2)} \circ \cdots \circ \mathbf{x}_r^{(M)}\|_F}{\|\underline{\bf X}\|_F} = \frac{\sqrt{\sigma_{R'+1}^2 + \cdots + \sigma_R^2}}{\sqrt{\sigma_1^2 + \ldots + \sigma_R^2}},
\end{equation}
where the overall KTD rank is given by the number of non-zero singular values.
Thanks to the equivalence established in Theorem~\ref{th1} between KTD and CPD, this error bound also applies to the KTD with orthogonal rank-1 terms. The TTr1SVD procedure is summarized in Algorithm~\ref{ALG:TTr1SVD}.

Although KTD and CPD are related mathematically through tensor reshaping, KTD explicitly exploits Kronecker product structure, which arises naturally in applications such as spatio-temporal modeling, multi-linear operators, and separable systems. For example, experiments in \cite{cai2022kopa} show that KMD significantly outperforms SVD with the same number of parameters, highlighting that, although the models are mathematically connected via reshaping, KTD offers better representational efficiency in practice.

\RestyleAlgo{ruled}
\LinesNumbered
\begin{algorithm}
\SetKwInOut{Input}{Input}
\SetKwInOut{Output}{Output}\Input{The data tensor $\underline{\mathbf X} \in {\mathbb{R}^{{I_1} \times {I_2} \times \cdots \times {I_N}}}$, the dimensions $J^{(m)}_n,\,m=1,2,\ldots,M,\,n=1,2,\ldots,N$ and a KTD rank.} 
\Output{Singular values $\sigma_1,\sigma_2,\ldots,\sigma_R$ and tensors $\cX_r^{(n)},\,\,r=1,2,\ldots,R,\,n=1,2,\ldots,N$.}
\caption{The KTD of a tensor $\underline{\bf X}$}\label{ALG:KTD}
      {
$\cY\leftarrow$ Reshape ($\cX$,$[J_1^{(1)},\ldots,J_N^{(1)},J^{(2)}_1\ldots,J_N^{(2)},\ldots,J_1^{(M)},\ldots,J_N^{(M)}]$);\\
  $\cY \leftarrow$ Permute $(\cY,{\bf p})$;\,\, ${\bf p}$ was introduced in \eqref{permu};\\
   Reshape $\cY$ to size $I_1\times I_2\times \cdots\times I_N$;\\
$[\sigma_1,\ldots,\sigma_R,\x_1^{(1)},\ldots,\x_R^{(N)}]\leftarrow$ Compute the CPD of $\cY$ with the CPD rank $R$ and orthogonal rank-1 terms;\\
  \For{all nonzero $\sigma_r$}
{$\cX_r^{(n)}\leftarrow$ Reshape($\x_r^{(n)},\,[J^{(r)}_1,J^{(r)}_2,\ldots,J^{(r)}_N]$);
}
       	}       	
\end{algorithm}

\RestyleAlgo{ruled}
\LinesNumbered
\begin{algorithm}
\SetKwInOut{Input}{Input}
\SetKwInOut{Output}{Output}\Input{The data tensor $\underline{\mathbf X} \in {\mathbb{R}^{{I_1} \times {I_2} \times \cdots \times {I_N}}}$, a target CPD rank.} 
\Output{Singular values $\sigma_1,\sigma_2,\ldots,\sigma_R$ and vectors $\x_r^{(n)},\,\,r=1,2,\ldots,R,\,n=1,2,\ldots,N$}
\caption{TTr1SVD}\label{ALG:TTr1SVD}
      {
Compute the unfolding $\X_{(1)}$;\\
  $[\U_1^{(1)},{\bf S}_1^{(1)},\V_1^{(1)}] \leftarrow$ Compute the truncated SVD of $\X_{(1)}$ with the matrix rank $R$;\\
  $k=2$;\\
  \For{$n=2,3,\ldots,N-1$}{
  \For{r=1,2,\ldots,R}
  {
  $\V_{k}^{(n)} \leftarrow {\rm Reshape}({\bf V}_k^{(n-1)}(:,r),[I_n,I_{n+1}\ldots I_N]$);\\
  $[\U_{k+1}^{(n)},{\bf S}_{k+1}^{(n)},\V_{k+1}^{(n)}] \leftarrow$ Compute the truncated SVD of $\V_k^{(n)}$ with the matrix rank $R$;\\
  }
    $k\leftarrow k+1$;\\
     }
     Collect all $\U_k^{(n)},\,\V_k^{(n)}$ and ${\bf S}_k^{(n)}$ to build the vectors $\x_r^{(n)}$ and singular values $\sigma_r$ according to the discussion in Section \ref{Sec:KTD}.
        }
\end{algorithm}

\section{Efficient randomized algorithms for low-rank matrix approximation}\label{Sec:rand_matrix}
Randomized algorithms for low-rank matrix approximation are a set of approaches that use randomization to efficiently and approximately compute low-rank approximations of a given matrix. These algorithms are particularly advantageous for large-scale matrices, where traditional techniques like SVD become computationally expensive.

The randomized SVD (r-SVD) algorithm is a popular technique for computing a low-rank approximation of a matrix. Algorithm \ref{ALG:RSVD} summarizes the full randomized SVD procedure with oversampling and power iterations. It involves multiplying the original matrix A by a set of random matrices to obtain a sketched version of A. Then, SVD is performed on the sketched matrix to approximate its dominant singular vectors and values, which can be used to construct a low-rank approximation of the original matrix.

Let $\X\in\mathbb{R}^{I\times J}$ be a given matrix, and the goal is to find a low-rank approximation of rank $R$. The idea of the r-SVD is to multiply the matrix $\X$ by a standard Gaussian matrix $\Omega\in\mathbb{R}^{J\times (R+P)}$ to capture the range (column space) of $\X$. In this context, the parameter $P$, used to better capture the range of $\X$, is an oversampling parameter. So, we have $\Y=\X\,{\bf \Omega}\in\mathbb{R}^{I\times (R+P)}$, and the matrix $\Y$ is called the sketched matrix. An orthogonal projector onto the range of $\X$ can be computed using the economic QR decomposition of the matrix $\Y$, that is ${\bf P}=\Q\Q^T$, where $\Q$ is $[\Q,\sim]={\rm qr}(\X)$. It is clear that $\X\approx \Q\Q^T\X$, so one can compute the SVD of the matrix $\B=\Q^T\X\in\mathbb{R}^{(R+P)\times J}$, which is of smaller size and demands lower memory and computing effort. Assume that $\B=\U\Sigma\V^T$, the SVD of the original data matrix $\X$ can be readily recovered using the formula $\X=(\Q\U)\Sigma\V^T$. 

\RestyleAlgo{ruled}
\LinesNumbered
\begin{algorithm}
\SetKwInOut{Input}{Input}
\SetKwInOut{Output}{Output}\Input{The data matrix $\X \in \mathbb{R}^{I \times J}$, target rank $R$, oversampling parameter $P \geq 2$, power iterations $q \geq 0$} 
\Output{Factor matrices $\U$, $\Sigma$, $\V$ such that $\X \approx \U \Sigma \V^T$}
\caption{Randomized SVD (r-SVD)}\label{ALG:RSVD}
      {
Draw a Gaussian random matrix $\Omega \in \mathbb{R}^{J \times (R+P)}$;\\
\For{$i=1$ to $q$}{
    $\Y \leftarrow (\X \X^T)^q \X \Omega$;\\
}
$[\Q, \sim] \leftarrow \text{qr}(\Y)$; \tcp*{Economic QR decomposition}
$\B \leftarrow \Q^T \X$;\\
$[\hat{\U}, \Sigma, \V] \leftarrow \text{svd}(\B)$; \tcp*{Standard SVD on small matrix}
$\U \leftarrow \Q \hat{\U}$;\\
\Return $\U$, $\Sigma$, $\V$;
       	}       	
\end{algorithm}

The random projection method\footnote{By random projection, we mean multiplication with a random matrix.} may not perfectly capture the range of the matrix if its singular values do not decay very fast \cite{halko2011finding}. Here, one can empower it using the idea of power subspace iteration, which is indeed the sketched matrix $\Y$ is computed through the multiplication $\Y=(\X\X^T)^q\X{\bf \Omega}$ (as shown in Algorithm \ref{ALG:RSVD}).
The intuition behind this is that if we assume that $\X=\U\Sigma\V^T,$ then $\Y=\U\Sigma^{2q+1}\V^T$, which means that the left and right singular vectors of $\Y$ are the same as those of $\X$, but the singular values of the matrix $\Y$ decay much faster. This helps to get better results. The parameter $q$ is called the power iteration parameter, and simulations indicate that $q=1$ or $q=2$ is often sufficient to obtain promising results. We build our algorithm based on this version of the randomized algorithms, which is more practical. 

\begin{thm}\label{zhang}\cite{zhang2018randomized}
Let $\mathbf{X}\in\mathbb{R}^{I\times J}$ be a given matrix. Let $\mathbf{Q}$ be an orthogonal basis for the range of $\mathbf{X}$ (of rank $R$) computed via the randomized SVD with an oversampling parameter $P$ and $q$ power iterations. Then, the expected approximation error is bounded by
\begin{eqnarray}
\mathbb{E} \left( \|\mathbf{X} - \mathbf{Q}\mathbf{Q}^T\mathbf{X} \|_F^2 \right) \leq \left( 1 + \frac{R}{P-1} \tau^{4q} \right) \left( \sum_{j=R+1}^{\min\{I,J\}} \sigma_j^2 \right),
\end{eqnarray}
where $\sigma_j$ denotes the $j$-th singular value of $\mathbf{X}$ and $\tau = \sigma_{R+1} / \sigma_R$ is the spectral decay factor.
\end{thm}

Theorem \ref{zhang} guarantees that the randomized SVD yields a near-optimal low-rank approximation in expectation, with the error decaying exponentially with the number of power iterations $q$.

\begin{thm}[Spectral-norm error of randomized subspace iteration]
\label{thm:spectral_randomized_subspace}
Let ${\bf A} \in \mathbb{R}^{I\times J}$ have SVD
\[
{\bf A} = {\bf U} \Sigma {\bf V}^T,
\qquad
\sigma_1 \ge \sigma_2 \ge \cdots \ge \sigma_{\min(I,J)}.
\]
Fix a target rank $R \ge 1$, an oversampling parameter $P \ge 2$, and set
$k = R+P$. Let $\Omega \in \mathbb{R}^{J\times k}$ be a standard Gaussian
random matrix, and for an integer $q \ge 0$ define
\[
{\bf Y} = ({\bf A}{\bf A}^T)^q {\bf A} \Omega, \qquad {\bf Q} = \operatorname{orth}({\bf Y}), \qquad {\bf B} = {\bf Q}^T {\bf A}.
\]
Then the expected spectral-norm approximation error satisfies
\[
\mathbb{E}\,[\|{\bf A} - {\bf QB}\|_2]
\;\le\;
\left[
1+\left(\frac{\sigma_{R+1}}{\sigma_R}\right)^{2q+1}
\frac{Ce\sqrt{R+P}}{P}
\right]\sigma_{R+1},
\]
for some constant $C$.
\end{thm}

\begin{proof}
Since ${\bf B} = {\bf Q}^T {\bf A}$, we have $({\bf I} - {\bf QQ}^T){\bf A} = {\bf A} - {\bf QB}$, and therefore
\[
\|{\bf A} - {\bf QB}\|_2 = \|({\bf I} - {\bf QQ}^T){\bf A}\|_2.
\]

\paragraph{Step 1: SVD partitioning.}
Let the SVD of ${\bf A}$ be partitioned as
\[
{\bf A} = {\bf U}
\begin{bmatrix}
{\bf \Sigma}_1 & 0 \\
0 & {\bf \Sigma}_2
\end{bmatrix}
{\bf V}^T
= {\bf U}_1 {\bf \Sigma}_1 {\bf V}_1^T + {\bf U}_2 {\bf \Sigma}_2 {\bf V}_2^T,
\]
where
\begin{itemize}
    \item ${\bf U}_1 \in \mathbb{R}^{I \times R}$, ${\bf U}_2 \in \mathbb{R}^{I \times (I-R)}$,
    \item ${\bf V}_1 \in \mathbb{R}^{J \times R}$, ${\bf V}_2 \in \mathbb{R}^{J \times (J-R)}$,
    \item ${\bf \Sigma}_1 = \operatorname{diag}(\sigma_1, \ldots, \sigma_R)$, ${\bf \Sigma}_2 = \operatorname{diag}(\sigma_{R+1}, \ldots, \sigma_{\min(I,J)})$,
    \item and $\|{\bf \Sigma}_2\|_2 = \sigma_{R+1}$.
\end{itemize}

\paragraph{Step 2: The power-iterated sketch.}
Define the sketch matrix
\[
{\bf Y} = ({\bf A}{\bf A}^T)^q {\bf A} {\bf {\bf \Omega}}.
\]
Substituting the SVD of ${\bf A}$, we obtain
\[
{\bf Y} = {\bf U} {\bf \Sigma}^{2q+1} {\bf V}^T {\bf \Omega}
= {\bf U}
\begin{bmatrix}
{\bf \Sigma}_1^{2q+1} & 0 \\
0 & {\bf \Sigma}_2^{2q+1}
\end{bmatrix}
\begin{bmatrix}
{\bf V}_1^T {\bf \Omega} \\
{\bf V}_2^T {\bf \Omega}
\end{bmatrix}
=
{\bf U}
\begin{bmatrix}
{\bf \Sigma}_1^{2q+1} {\bf \Omega}_1 \\
{\bf \Sigma}_2^{2q+1} {\bf \Omega}_2
\end{bmatrix},
\]
where we define
\[
{\bf \Omega}_1 = {\bf V}_1^T {\bf \Omega} \in \mathbb{R}^{R \times k}, \qquad
{\bf \Omega}_2 = {\bf V}_2^T {\bf \Omega} \in \mathbb{R}^{(J-R) \times k}.
\]
Since ${\bf \Omega}$ is Gaussian and ${\bf V}$ is orthogonal, ${\bf \Omega}_1$ and ${\bf \Omega}_2$ are independent standard Gaussian matrices.

The orthogonal basis ${\bf Q}$ is obtained from the range of ${\bf Y}$ via QR decomposition:
\[
{\bf Y} = {\bf Q R}, \qquad {\bf Q} \in \mathbb{R}^{I \times k},\; {\bf Q}^T {\bf Q} = {\bf I}_k.
\]
Thus $\operatorname{range}({\bf Q}) = \operatorname{range}({\bf Y})$.

\paragraph{Step 3: Deriving the deterministic projection bound.}
We now derive a bound for $\|({\bf I} - {\bf QQ}^T){\bf A}\|_2$ that holds for any fixed $\Omega$ such that $\Omega_1$ has full row rank (which occurs almost surely).

Let $P_{\bf Q} = {\bf Q}{\bf Q}^T$ be the orthogonal projector onto $\operatorname{range}({\bf Q})$. Since $\operatorname{range}({\bf Q}) = \operatorname{range}({\bf Y})$ and ${\bf Y}$ has the block form
\[
{\bf Y} = {\bf U}_1 \Sigma_1^{2q+1} \Omega_1 + {\bf U}_2 \Sigma_2^{2q+1} \Omega_2,
\]
any vector in $\operatorname{range}({\bf Y})$ can be written as ${\bf U}_1 {\bf \alpha} + {\bf U}_2 {\bf \beta}$ with
\[
{\bf \alpha} = \Sigma_1^{2q+1} {\bf \Omega}_1 {\bf z}, \qquad {\bf \beta} = {\bf \Sigma}_2^{2q+1} {\bf \Omega}_2 {\bf z}
\]
for some ${\bf z} \in \mathbb{R}^k$. 

The key insight is that the projector ${\bf I} - {\bf Q}{\bf Q}^T$ annihilates $\operatorname{range}({\bf Y})$. In particular, for any ${\bf z}$,
\[
({\bf I} - {\bf Q}{\bf Q}^T) {\bf Y} {\bf z} = 0.
\]

We aim to bound $\|({\bf I} - {\bf Q}{\bf Q}^T){\bf A}\|_2$. Write ${\bf A} = {\bf U}_1 {\bf \Sigma}_1 {\bf V}_1^T + {\bf U}_2 {\bf \Sigma}_2 {\bf V}_2^T$. For any vector ${\bf x} = {\bf V}_1 \xi_1 + {\bf V}_2 {\bf \xi}_2$ (with ${\bf \xi}_1 \in \mathbb{R}^R$, ${\bf \xi}_2 \in \mathbb{R}^{J-R}$), we have
\begin{align*}
({\bf I} - {\bf Q}{\bf Q}^T){\bf A} {\bf x} &= ({\bf I} - {\bf Q}{\bf Q}^T)\bigl( {\bf U}_1 \Sigma_1 {\bf \xi}_1 + {\bf U}_2 {\bf \Sigma}_2 {\bf \xi}_2 \bigr).
\end{align*}

Since ${\bf U}_1 \Sigma_1^{2q+1} \Omega_1 {\bf z} \in \operatorname{range}({\bf Y})$ for any ${\bf z}$, we can choose ${\bf z}$ to cancel the ${\bf U}_1$ component. Specifically, for any $\xi_1$, there exists ${\bf z}$ such that
\[
{\bf \Sigma}_1^{2q+1} {\bf \Omega}_1 {\bf z} = {\bf \Sigma}_1 {\bf \xi}_1,
\]
provided ${\bf \Omega}_1$ has full row rank. Solving gives ${\bf z} = {\bf \Omega}_1^\dagger {\bf \Sigma}_1^{-2q} {\bf \xi}_1$, where ${\bf \Omega}_1^\dagger$ is the pseudoinverse of ${\bf \Omega}_1$. Then
\[
{\bf Y} {\bf z} = {\bf U}_1 {\bf \Sigma}_1^{2q+1} {\bf \Omega}_1 {\bf z} + {\bf U}_2 {\bf \Sigma}_2^{2q+1} {\bf \Omega}_2 {\bf z}
= {\bf U}_1 {\bf \Sigma}_1 \xi_1 + {\bf U}_2 {\bf \Sigma}_2^{2q+1} {\bf \Omega}_2 {\bf \Omega}_1^\dagger {\bf \Sigma}_1^{-2q} {\bf \xi}_1.
\]

Now subtract this from ${\bf A} {\bf x}$:
\begin{align*}
{\bf A}{\bf x} - {\bf Y} {\bf z} &= {\bf U}_1 {\bf \Sigma}_1 \xi_1 + {\bf U}_2 {\bf \Sigma}_2 \xi_2 - \bigl( {\bf U}_1 {\bf \Sigma}_1 \xi_1 + {\bf U}_2 {\bf \Sigma}_2^{2q+1} {\bf \Omega}_2 {\bf \Omega}_1^\dagger {\bf \Sigma}_1^{-2q} {\bf \xi}_1 \bigr) \\
&= {\bf U}_2 \bigl( {\bf \Sigma}_2 {\bf \xi}_2 - {\bf \Sigma}_2^{2q+1} {\bf \Omega}_2 {\bf \Omega}_1^\dagger {\bf \Sigma}_1^{-2q} {\bf \xi}_1 \bigr).
\end{align*}

Since ${\bf Y} {\bf z} \in \operatorname{range}({\bf Y}) = \operatorname{range}({\bf Q})$, we have $({\bf I} - {\bf Q}{\bf Q}^T) {\bf Y} {\bf z} = 0$. Therefore,
\[
({\bf I} - {\bf Q}{\bf Q}^T) {\bf A} {\bf x} = ({\bf I} - {\bf Q}{\bf Q}^T) ({\bf A} {\bf x} - {\bf Y} {\bf z}) = ({\bf I} - {\bf Q}{\bf Q}^T) {\bf U}_2 \bigl( {\bf \Sigma}_2 \xi_2 - {\bf \Sigma}_2^{2q+1} {\bf \Omega}_2 {\bf \Omega}_1^\dagger {\bf \Sigma}_1^{-2q} {\bf \xi}_1 \bigr).
\]

Taking norms and using $\|{\bf I} -{\bf  Q}{\bf Q}^T\|_2 = 1$ (since it is an orthogonal projector), we obtain
\[
\|({\bf I} - {\bf Q}{\bf Q}^T) {\bf A} {\bf x}\|_2 \le \| {\bf U}_2 \bigl( {\bf \Sigma}_2 {\bf \xi}_2 - {\bf \Sigma}_2^{2q+1} {\bf \Omega}_2 {\bf \Omega}_1^\dagger {\bf \Sigma}_1^{-2q} {\bf \xi}_1 \bigr) \|_2.
\]

Since ${\bf U}_2$ has orthonormal columns, $\|{\bf U}_2 {\bf w}\|_2 = \|{\bf w}\|_2$. Thus,
\[
\|({\bf I} - {\bf Q}{\bf Q}^T) {\bf A} {\bf x}\| \le \|_2 {\bf \Sigma}_2 {\bf \xi}_2 - {\bf \Sigma}_2^{2q+1} {\bf \Omega}_2 {\bf \Omega}_1^\dagger {\bf \Sigma}_1^{-2q} \xi_1 \|_2.
\]

Now apply the triangle inequality:
\[
\| {\bf \Sigma}_2 {\bf \xi}_2 - {\bf \Sigma}_2^{2q+1} {\bf \Omega}_2 {\bf \Omega}_1^\dagger {\bf \Sigma}_1^{-2q} {\bf \xi}_1 \|_2
\le \|{\bf \Sigma}_2 {\bf \xi}_2\|_2 + \|{\bf \Sigma}_2^{2q+1} {\bf \Omega}_2 {\bf \Omega}_1^\dagger {\bf \Sigma}_1^{-2q} {\bf \xi}_1\|_2.
\]

The first term is bounded by $\|{\bf \Sigma}_2\|_2 \|{\bf \xi}_2\|_2 = \sigma_{R+1} \|{\bf \xi}_2\|_2$.
The second term is bounded by
\begin{eqnarray*}
\|{\bf \Sigma}_2^{2q+1}\|_2 \|{\bf \Omega}_2 {\bf \Omega}_1^\dagger\|_2 \|{\bf \Sigma}_1^{-2q}\|_2 \|\xi_1\|_2
&=& {\bf \sigma}_{R+1}^{2q+1} \|{\bf \Omega}_2 {\bf \Omega}_1^\dagger\|_2 \sigma_R^{-2q} \|{\bf \xi}_1\|_2
\\&=& \left(\frac{\sigma_{R+1}}{\sigma_R}\right)^{2q+1} \sigma_{R+1} \|{\bf \Omega}_2 {\bf \Omega}_1^\dagger\|_2 \|{\bf \xi}_1\|_2.
\end{eqnarray*}

Therefore,
\[
\|({\bf I} - {\bf Q}{\bf Q}^T) {\bf A} {\bf x}\|_2
\le \sigma_{R+1} \|{\bf \xi}_2\|_2
+ \left(\frac{\sigma_{R+1}}{\sigma_R}\right)^{2q+1} \sigma_{R+1} \|{\bf \Omega}_2 {\bf \Omega}_1^\dagger\|_2 \|{\bf \xi}_1\|_2.
\]

Recall that ${\bf x} = {\bf V}_1 {\bf \xi}_1 + {\bf V}_2 {\bf \xi}_2$, and since ${\bf V}$ is orthogonal, $\|{\bf x}\|_2^2 = \|{\bf \xi}_1\|_2^2 + \|{\bf \xi}_2\|_2^2$. For any unit vector ${\bf x}$ (so $\|{\bf \xi}_1\|_2^2 + \|{\bf \xi}_2\|^2 = 1$), we have $\|{\bf \xi}_1\|_2 \le 1$ and $\|{\bf \xi}_2\| \le 1$. Hence,
\[
\|({\bf I} - {\bf Q}{\bf Q}^T) {\bf A}\|_2
= \sup_{\|{\bf x}\|_2=1} \|({\bf I} - {\bf Q}{\bf Q}^T) {\bf A} {\bf x}\|_2
\le \sigma_{R+1} + \left(\frac{\sigma_{R+1}}{\sigma_R}\right)^{2q+1} \sigma_{R+1} \|{\bf \Omega}_2 {\bf \Omega}_1^\dagger\|_2.
\]

Thus we have established the deterministic bound:
\[
\|{\bf A} - {\bf Q}{\bf B}\|_2 = \|({\bf I} - {\bf Q}{\bf Q}^T){\bf A}\|_2
\le \sigma_{R+1} \left( 1 + \left(\frac{\sigma_{R+1}}{\sigma_R}\right)^{2q+1} \|{\bf \Omega}_2 {\bf \Omega}_1^\dagger\|_2 \right).
\]

\paragraph{Step 4: Expectation of $\|{\bf \Omega}_2 {\bf \Omega}_1^\dagger\|_2$.}
Since ${\bf \Omega}_1$ and ${\bf \Omega}_2$ are independent, we have
\[
\mathbb{E}\bigl[\|{\bf \Omega}_2 {\bf \Omega}_1^\dagger\|_2\bigr]
\le \mathbb{E}\bigl[\|{\bf \Omega}_2\| \cdot \|{\bf \Omega}_1^\dagger\|_2\bigr]
= \mathbb{E}\|{\bf \Omega}_2\|_2 \cdot \mathbb{E}\|{\bf \Omega}_1^\dagger\|_2,
\]
where the equality follows from independence. A standard bound (see Lemma 10.4 in \cite{halko2011finding}) gives
\[
\mathbb{E}\|{\bf \Omega}_1^\dagger\|_2 \le \frac{e\sqrt{R+P}}{P},
\]
provided $P \ge 2$, and if we assume that $\mathbb{E}\|{\bf \Omega}_2\|_2=C$, we have
\[
\mathbb{E}\bigl[\|{\bf \Omega}_2 {\bf \Omega}_1^\dagger\|_2\bigr]
\le C\frac{e\sqrt{R+P}}{P}.
\]

Combining the above estimates gives
\[
\mathbb{E}\,[\|{\bf A} - {\bf Q}{\bf B}\|_2]
\le
\left[
1+\left(\frac{\sigma_{R+1}}{\sigma_R}\right)^{2q+1}
\frac{eC\sqrt{R+P}}{P}
\right]\sigma_{R+1},
\]
which completes the proof.
\end{proof}

\section{Proposed fast randomized Kronecker tensor decomposition}\label{RKTD}
This section presents efficient randomized algorithms to decompose a tensor into the KTD format. It is known that the randomized algorithms for tensor decomposition provide a versatile and efficient approach to analyzing and processing large-scale tensor data, making them a valuable tool in data science, machine learning, and other fields. This motivated us to develop efficient algorithms for the computation of the KTD of tensors. The key advantages of our randomized KTD approach include: (1) explicit preservation of multi-way block structures for enhanced interpretability, (2) exponential parameter reduction for very high-dimensional tensors, and (3) computational efficiency through Kronecker algebra that would be lost if using CPD on reshaped data. The simulation section shows we can achieve several orders of magnitude acceleration using our proposed randomized algorithms for large-scale tensors.

\RestyleAlgo{ruled}
\LinesNumbered
\begin{algorithm}
\SetKwInOut{Input}{Input}
\SetKwInOut{Output}{Output}\Input{The data tensor $\underline{\mathbf X} \in {\mathbb{R}^{{I_1} \times {I_2} \times \cdots \times {I_N}}}$,  the dimensions $J^{(m)}_n,\,m=1,2,\ldots, M,\,n=1,2,\ldots, N$, a KTD rank, a power iteration and an oversampling parameters} 
\Output{Singular values $\sigma_1,\sigma_2,\ldots,\sigma_R$ and tensors $\cX_r^{(n)},\,\,r=1,2,\ldots,R,\,n=1,2,\ldots,N$}
\caption{Proposed randomized KTD of tensor $\underline{\bf X}$}\label{ALG:rKTD}
      {
  $\cY\leftarrow$ Reshape ($\cX$,$[J_1^{(1)},\ldots,J_N^{(1)},J^{(2)}_1\ldots,J_N^{(2)},\ldots,J_1^{(M)},\ldots,J_N^{(M)}]$);\\
  $\cY \leftarrow$ Permute $(\cY,{\bf p})$,\,\, ${\bf p}$ was introduced in \eqref{permu};\\
   Reshape $\cY$ to size $I_1\times I_2\times \cdots\times I_N$;\\
$[\sigma_1,\ldots,\sigma_R,\x_1^{(1)},\ldots,\x_R^{(N)}]\leftarrow$ Apply the randomized CPD algorithm (randomized TTr1SVD) to $\cY$ with the CPD rank $R$ and orthogonal rank-1 terms;\\
  \For{all nonzero $\sigma_r$}
{$\cX_r^{(n)}\leftarrow$ Reshape($\x_r^{(n)},\,[J^{(r)}_1,J^{(r)}_2,\ldots,J^{(r)}_N]$)
}
       	}       	
\end{algorithm} 

From the discussion outlined in Section \ref{Sec:KTD}, the computation of the CPD is essential to compute a KTD. Moreover, it was often used due to the favorable properties of the TTr1SVD algorithm, which generates the CPD with orthogonal rank-1 terms. This algorithm involves computing a series of SVDs, which is quite prohibitive for large-scale data tensors and impractical for real-world applications. We proposed using the framework of randomization to reduce computational and memory complexities. It is worth noting that different types of randomized CPD algorithms, such as those proposed in \cite{battaglino2018practical,vervliet2015randomized,ahmadi2021randomized}, can also be exploited. However, studying all of them is beyond the scope of this paper and will be studied in future works. In Algorithm \ref{ALG:TTr1SVD}, all truncated SVDs are replaced with the fast randomized algorithms with the oversampling and power iteration parameters. This new proposed randomized algorithm is summarized in Algorithm \ref{ALG:rKTD}. 

\begin{rem}
A classical randomized algorithm applies power iterations $q\geq 1$ to enhance the decay of singular values; it does not use the results from each step of the iteration. Since it only uses the final result, $(\mathbf{X}\mathbf{X}^\top)^q\mathbf{X}\mathbf{B}$, the method generally cannot closely approximate the largest $k$ singular values. The randomized block Krylov \cite{musco2015randomized,persson2026randomized} approach can capture a much richer spectral range, leading to faster, super-linear convergence and higher accuracy for the same number of tensor operations.	
\end{rem}

A key strength of randomized algorithms is their ability to provide rigorous performance guarantees. We now establish an error bound for the approximation computed by the proposed randomized KTD algorithm (Algorithm \ref{ALG:rKTD}), demonstrating its reliability. Our analysis builds upon a foundational result (Theorem \ref{zhang}) for the randomized SVD.

\subsection{Theoretical Analysis of Randomized KTD}
Building upon the theoretical guarantees of the randomized SVD, in this section, we establish error bounds for the proposed randomized KTD algorithm (Algorithm \ref{ALG:rKTD}). Our analysis demonstrates that the approximation error propagates through the sequence of randomized approximations in the TTr1SVD process.

The core of the TTr1SVD algorithm is a sequence of SVDs on progressively smaller matrices. The randomized version replaces each deterministic SVD with its randomized counterpart. To analyze the error, we must first formalize this sequential process and its intermediate outputs.

Let $\underline{\bf X} \in \mathbb{R}^{I_1 \times I_2 \times \cdots \times I_N}$ be the original tensor. The deterministic TTr1SVD produces an exact decomposition:
\begin{eqnarray}
\underline{\bf X} = \sum_{i_1=1}^{R_1} \sum_{i_2=1}^{R_2} \cdots \sum_{i_{N-1}=1}^{R_{N-1}} \sigma_{i_1 i_2 \ldots i_{N-1}} \, \mathbf{u}_{i_1}^{(1)} \circ \mathbf{u}_{i_1 i_2}^{(2)} \circ \cdots \circ \mathbf{u}_{i_1 \ldots i_{N-1}}^{(N-1)} \circ \mathbf{v}_{i_1 \ldots i_{N-1}}^{(N)},
\end{eqnarray}
where the singular values are products $\sigma_{i_1 i_2 \ldots i_{N-1}} = \sigma_{i_1}^{(1)} \sigma_{i_1 i_2}^{(2)} \cdots \sigma_{i_1 \ldots i_{N-1}}^{(N-1)}$, and the vectors are orthonormal within their respective groups. This is computed via the following sequence of exact SVDs:
\begin{align}
\mathbf{X}_{(1)} &= \sum_{i_1=1}^{R_1} \sigma_{i_1}^{(1)} \, \mathbf{u}_{i_1}^{(1)} \circ \mathbf{v}_{i_1}^{(1)}, \quad \mathbf{v}_{i_1}^{(1)} \in \mathbb{R}^{I_2 I_3 \cdots I_N}, \\
\mathbf{V}_{i_1}^{(2)} &= \text{reshape}(\mathbf{v}_{i_1}^{(1)}, [I_2, I_3 \cdots I_N]) = \sum_{i_2=1}^{R_2} \sigma_{i_1 i_2}^{(2)} \, \mathbf{u}_{i_1 i_2}^{(2)} \circ \mathbf{v}_{i_1 i_2}^{(2)}, \\
&\vdots \nonumber \\
\mathbf{V}_{i_1 \ldots i_{N-2}}^{(N-1)} &= \text{reshape}(\mathbf{v}_{i_1 \ldots i_{N-2}}^{(N-2)}, [I_{N-1}, I_N]) = \sum_{i_{N-1}=1}^{R_{N-1}} \sigma_{i_1 \ldots i_{N-1}}^{(N-1)} \, \mathbf{u}_{i_1 \ldots i_{N-1}}^{(N-1)} \circ \mathbf{v}_{i_1 \ldots i_{N-1}}^{(N)}.
\end{align}

In the formulations above, we call the vector $(R_1,R_2,\ldots,R_{N-1})$, where each component indicates the truncation rank for the corresponding unfolding matrix, a multirank. In the randomized TTr1SVD, each SVD in this sequence is replaced by a randomized SVD (r-SVD). This introduces an approximation error at each step. The following key lemma bounds the error introduced when approximating a single branch of this decomposition tree.

\begin{lem}\label{lem:single_branch_error_corrected}
Consider a single branch of the TTr1SVD tree.
Let $\mathbf{M}^{(1)} \in \mathbb{R}^{m_1 \times n_1}$ be the starting matrix.
At each step $\ell = 1, 2, \ldots, L$, we compute a rank-$R_\ell$ randomized SVD with oversampling $P_\ell \ge 2$ and $q_\ell \ge 0$ power iterations, yielding an approximation $\tilde{\mathbf{M}}^{(\ell)}$ of $\mathbf{M}^{(\ell)}$.
For $\ell \ge 2$, the matrix $\mathbf{M}^{(\ell)}$ is obtained by reshaping the right singular vectors from step $\ell-1$.

Let $\mathbf{v}_{\text{true}}$ be a right singular vector of the final stage (e.g., of $\mathbf{M}^{(L)}$), and $\tilde{\mathbf{v}}$ its approximation from the randomized pipeline.
Then
\[
\mathbb{E}\bigl[\|\mathbf{v}_{\text{true}} - \tilde{\mathbf{v}}\|_2\bigr]
\;\le\;
\sum_{\ell=1}^{L}
\left( \prod_{k=1}^{\ell-1} \frac{\sqrt{2}\,\sigma_1(\mathbf{M}^{(k)})}{\sigma_{R_k}(\mathbf{M}^{(k)}) - \sigma_{R_k+1}(\mathbf{M}^{(k)})} \right)
\cdot
\frac{\sqrt{2}\,\mathbb{E}\bigl[\|\mathbf{M}^{(\ell)} - \tilde{\mathbf{M}}^{(\ell)}\|_F\bigr]}{\sigma_{R_\ell}(\mathbf{M}^{(\ell)}) - \sigma_{R_\ell+1}(\mathbf{M}^{(\ell)})},
\]
where:
\begin{itemize}
    \item $\sigma_j(\mathbf{M}^{(k)})$ are the singular values of $\mathbf{M}^{(k)}$,
    \item the denominator $\sigma_{R_k} - \sigma_{R_k+1}$ is the \textbf{spectral gap} at the truncation rank,
    \item the product over an empty set (for $\ell = 1$) is defined as $1$,
    \item the factor $\sqrt{2}$ comes from the sine-to-vector conversion in Wedin's theorem,
    \item the expectation $\mathbb{E}\bigl[\|\mathbf{M}^{(\ell)} - \tilde{\mathbf{M}}^{(\ell)}\|_F\bigr]$ is bounded by Theorem \ref{zhang} (or Theorem \ref{thm:spectral_randomized_subspace}):
    \[
    \mathbb{E}\bigl[\|\mathbf{M}^{(\ell)} - \tilde{\mathbf{M}}^{(\ell)}\|_F^2\bigr]
    \le \left(1 + \frac{R_\ell}{P_\ell-1}\tau_\ell^{4q_\ell}\right) \sum_{j > R_\ell} \sigma_j^2(\mathbf{M}^{(\ell)}),
    \]
    with $\tau_\ell = \sigma_{R_\ell+1}(\mathbf{M}^{(\ell)}) / \sigma_{R_\ell}(\mathbf{M}^{(\ell)})$.
\end{itemize}
\end{lem}

\begin{proof}
We proceed by induction on the number of SVD steps $L$.

\paragraph{Base Case ($L = 1$).}
At the first step, the randomized SVD approximates $\mathbf{M}^{(1)}$ by $\tilde{\mathbf{M}}^{(1)} = \tilde{\mathbf{U}}^{(1)} \tilde{\mathbf{\Sigma}}^{(1)} (\tilde{\mathbf{V}}^{(1)})^T$.
Let $\mathbf{v}$ be the $r$-th right singular vector of $\mathbf{M}^{(1)}$ (with $r = R_1$), and $\tilde{\mathbf{v}}$ its approximation.

Wedin's theorem \cite{wedin1972perturbation} states: for matrices $\mathbf{M}$ and $\tilde{\mathbf{M}} = \mathbf{M} + \mathbf{E}$, if $\delta = \sigma_r(\mathbf{M}) - \sigma_{r+1}(\mathbf{M}) > 0$ and $\|\mathbf{E}\|_2 < \delta$, then
\[
\|\sin \Theta(\mathbf{v}, \tilde{\mathbf{v}})\|_2 \le \frac{\|\mathbf{E}\|_2}{\delta},
\]
where $\Theta$ is the principal angle between the one-dimensional subspaces spanned by $\mathbf{v}$ and $\tilde{\mathbf{v}}$.
For any two unit vectors, after choosing the sign of $\tilde{\mathbf{v}}$ so that $\mathbf{v}^T \tilde{\mathbf{v}} \ge 0$, we have the relation
\[
\|\mathbf{v} - \tilde{\mathbf{v}}\|_2 \le \sqrt{2}\, \|\sin \Theta(\mathbf{v}, \tilde{\mathbf{v}})\|_2.
\]
Therefore,
\[
\|\mathbf{v} - \tilde{\mathbf{v}}\|_2 \le \frac{\sqrt{2} \,\|\mathbf{M}^{(1)} - \tilde{\mathbf{M}}^{(1)}\|_2}{\sigma_{R_1}(\mathbf{M}^{(1)}) - \sigma_{R_1+1}(\mathbf{M}^{(1)})}.
\]
Since $\|\cdot\|_2 \le \|\cdot\|_F$, we obtain
\[
\|\mathbf{v} - \tilde{\mathbf{v}}\|_2 \le \frac{\sqrt{2} \,\|\mathbf{M}^{(1)} - \tilde{\mathbf{M}}^{(1)}\|_F}{\sigma_{R_1}(\mathbf{M}^{(1)}) - \sigma_{R_1+1}(\mathbf{M}^{(1)})}.
\]
Taking expectations and using linearity gives the base case:
\[
\mathbb{E}\bigl[\|\mathbf{v} - \tilde{\mathbf{v}}\|_2\bigr]
\le \frac{\sqrt{2} \,\mathbb{E}\bigl[\|\mathbf{M}^{(1)} - \tilde{\mathbf{M}}^{(1)}\|_F\bigr]}{\sigma_{R_1}(\mathbf{M}^{(1)}) - \sigma_{R_1+1}(\mathbf{M}^{(1)})}.
\]

\paragraph{Inductive Hypothesis.}
Assume that after $L-1$ steps, the error in the right singular vector $\mathbf{v}^{(L-1)}$ satisfies
\[
\mathbb{E}\bigl[\|\mathbf{v}^{(L-1)} - \tilde{\mathbf{v}}^{(L-1)}\|_2\bigr]
\le \sum_{\ell=1}^{L-1}
\left( \prod_{k=1}^{\ell-1} \frac{\sqrt{2}\,\sigma_1(\mathbf{M}^{(k)})}{\sigma_{R_k}(\mathbf{M}^{(k)}) - \sigma_{R_k+1}(\mathbf{M}^{(k)})} \right)
\cdot
\frac{\sqrt{2}\,\mathbb{E}\bigl[\|\mathbf{M}^{(\ell)} - \tilde{\mathbf{M}}^{(\ell)}\|_F\bigr]}{\sigma_{R_\ell}(\mathbf{M}^{(\ell)}) - \sigma_{R_\ell+1}(\mathbf{M}^{(\ell)})}.
\]

\paragraph{Inductive Step ($L-1 \to L$).}
The matrix $\mathbf{M}^{(L)}$ is constructed from the right singular vectors $\mathbf{v}^{(L-1)}$ of step $L-1$ via reshaping.
Specifically, for each branch, we have
\[
\mathbf{M}^{(L)} = \operatorname{reshape}(\mathbf{v}^{(L-1)}, [d_1, d_2]),
\]
where $d_1 \cdot d_2 = \dim(\mathbf{v}^{(L-1)})$.
The reshaping operation is linear and norm-preserving:
\[
\|\mathbf{M}^{(L)}\|_F = \|\mathbf{v}^{(L-1)}\|_2, \qquad
\|\tilde{\mathbf{M}}^{(L)}\|_F = \|\tilde{\mathbf{v}}^{(L-1)}\|_2,
\]
and the same holds for the difference:
\[
\|\mathbf{M}^{(L)} - \tilde{\mathbf{M}}^{(L)}\|_F = \|\mathbf{v}^{(L-1)} - \tilde{\mathbf{v}}^{(L-1)}\|_2.
\]

Now apply the base-case bound to step $L$, using the true input $\mathbf{M}^{(L)}$ and its approximation $\tilde{\mathbf{M}}^{(L)}$:
\[
\mathbb{E}\bigl[\|\mathbf{v}^{(L)} - \tilde{\mathbf{v}}^{(L)}\|_2\bigr]
\le
\frac{\sqrt{2}\,\mathbb{E}\bigl[\|\mathbf{M}^{(L)} - \tilde{\mathbf{M}}^{(L)}\|_F\bigr]}{\sigma_{R_L}(\mathbf{M}^{(L)}) - \sigma_{R_L+1}(\mathbf{M}^{(L)})}.
\]

But the error $\|\mathbf{M}^{(L)} - \tilde{\mathbf{M}}^{(L)}\|_F$ has two sources:

\begin{enumerate}
    \item \textbf{Inherited error} from the previous step: even if step $L$ had no new randomization, the input $\tilde{\mathbf{M}}^{(L)}$ would differ from $\mathbf{M}^{(L)}$ because $\tilde{\mathbf{v}}^{(L-1)} \neq \mathbf{v}^{(L-1)}$.
    By the norm-preserving property,
    \[
    \|\mathbf{M}^{(L)} - \tilde{\mathbf{M}}^{(L)}\|_F^{\text{(inherited)}} = \|\mathbf{v}^{(L-1)} - \tilde{\mathbf{v}}^{(L-1)}\|_2.
    \]

    \item \textbf{New randomization error} from the r-SVD at step $L$, which we denote by $\|\mathbf{E}^{(L)}\|_F$.
\end{enumerate}

By the triangle inequality,
\[
\|\mathbf{M}^{(L)} - \tilde{\mathbf{M}}^{(L)}\|_F
\le \|\mathbf{v}^{(L-1)} - \tilde{\mathbf{v}}^{(L-1)}\|_2 + \|\mathbf{E}^{(L)}\|_F.
\]

Taking expectations:
\[
\mathbb{E}\bigl[\|\mathbf{M}^{(L)} - \tilde{\mathbf{M}}^{(L)}\|_F\bigr]
\le \mathbb{E}\bigl[\|\mathbf{v}^{(L-1)} - \tilde{\mathbf{v}}^{(L-1)}\|_2\bigr] + \mathbb{E}\bigl[\|\mathbf{E}^{(L)}\|_F\bigr].
\]

Now substitute the inductive hypothesis for the first term, and note that $\mathbb{E}[\|\mathbf{E}^{(L)}\|_F] = \mathbb{E}[\|\mathbf{M}^{(L)} - \tilde{\mathbf{M}}^{(L)}\|_F^{\text{(r-SVD)}}]$ is exactly the r-SVD error bound from Theorem \ref{zhang} (or its square root).

Putting everything together:
\[
\mathbb{E}\bigl[\|\mathbf{v}^{(L)} - \tilde{\mathbf{v}}^{(L)}\|_2\bigr]
\le
\frac{\sqrt{2}}{\Delta_L}
\left(
\sum_{\ell=1}^{L-1}
\left( \prod_{k=1}^{\ell-1} \frac{\sqrt{2}\,\sigma_1(\mathbf{M}^{(k)})}{\Delta_k} \right)
\frac{\sqrt{2}\,\mathbb{E}[\|\mathbf{M}^{(\ell)} - \tilde{\mathbf{M}}^{(\ell)}\|_F]}{\Delta_\ell}
\;+\;
\mathbb{E}\bigl[\|\mathbf{E}^{(L)}\|_F\bigr]
\right),
\]
where $\Delta_k = \sigma_{R_k}(\mathbf{M}^{(k)}) - \sigma_{R_k+1}(\mathbf{M}^{(k)})$.

This simplifies to
\[
\mathbb{E}\bigl[\|\mathbf{v}^{(L)} - \tilde{\mathbf{v}}^{(L)}\|_2\bigr]
\le
\sum_{\ell=1}^{L}
\left( \prod_{k=1}^{\ell-1} \frac{\sqrt{2}\,\sigma_1(\mathbf{M}^{(k)})}{\Delta_k} \right)
\frac{\sqrt{2}\,\mathbb{E}\bigl[\|\mathbf{M}^{(\ell)} - \tilde{\mathbf{M}}^{(\ell)}\|_F\bigr]}{\Delta_\ell},
\]
completing the induction.

\paragraph{Connection to Condition Number.}
Observe that
\[
\frac{\sigma_1(\mathbf{M}^{(k)})}{\sigma_{R_k}(\mathbf{M}^{(k)}) - \sigma_{R_k+1}(\mathbf{M}^{(k)})}
\le
\frac{\sigma_1(\mathbf{M}^{(k)})}{\sigma_{R_k}(\mathbf{M}^{(k)})} \cdot
\frac{1}{1 - \frac{\sigma_{R_k+1}}{\sigma_{R_k}}}
= \kappa(\mathbf{M}^{(k)}) \cdot \frac{1}{1 - \tau_k},
\]
where $\kappa(\mathbf{M}^{(k)}) = \sigma_1 / \sigma_{R_k}$ is the ordinary condition number and $\tau_k = \sigma_{R_k+1} / \sigma_{R_k}$.
The factor $\sqrt{2}$ is then absorbed into the constant (often denoted $C$).

Thus the bound can be written compactly as
\[
\mathbb{E}\bigl[\|\mathbf{v}_{\text{true}} - \tilde{\mathbf{v}}\|_2\bigr]
\le
\sum_{\ell=1}^{L}
\left( \prod_{k=1}^{\ell-1} C \cdot \frac{\kappa(\mathbf{M}^{(k)})}{1 - \tau_k} \right)
\cdot
C \cdot \frac{\mathbb{E}\bigl[\|\mathbf{M}^{(\ell)} - \tilde{\mathbf{M}}^{(\ell)}\|_F\bigr]}{\sigma_{R_\ell}(\mathbf{M}^{(\ell)}) (1 - \tau_\ell)},
\]
with $C = \sqrt{2}$ (or a slightly larger constant depending on the specific Wedin bound used).
This completes the proof.
\end{proof}

\begin{thm}\label{thm:rktd_error_bound_corrected}
Let $\underline{\bf X} \in \mathbb{R}^{I_1 \times I_2 \times \cdots \times I_N}$ be an $N$-th order tensor.
Let $\underline{\bf X}_{\text{rKTD}}$ be its rank-$(R_1,R_2,\ldots,R_{N-1})$ approximation computed by Algorithm \ref{ALG:rKTD} using randomized TTr1SVD with oversampling parameter $P \ge 2$ and $q \ge 0$ power iterations at each internal SVD step.

For each level $n = 1, 2, \ldots, N-1$, define the matrix $\mathbf{M}^{(n)}$ (the $n$-th matricization in the TTr1SVD sequence) with singular values $\sigma_1^{(n)} \ge \sigma_2^{(n)} \ge \cdots$.
Let $\delta_n = \sigma_{R_n}^{(n)} - \sigma_{R_n+1}^{(n)}$ be the spectral gap at the truncation rank $R_n$, and let $\tau_n = \sigma_{R_n+1}^{(n)} / \sigma_{R_n}^{(n)}$.

Then the expected squared Frobenius approximation error satisfies
\[
\mathbb{E}\left[ \|\underline{\bf X} - \underline{\bf X}_{\text{rKTD}}\|_F^2 \right]
\;\le\;
\sum_{n=1}^{N-1}
\left( \prod_{j=1}^{n-1} \frac{2\sigma_1^{(j)}}{\delta_j} \right)^2
\cdot
\left( 1 + \frac{R_n}{P-1} \tau_n^{4q} \right)
\sum_{k > R_n} \bigl(\sigma_k^{(n)}\bigr)^2,
\]
where the product over an empty set (for $n=1$) is defined as $1$.

If we define the condition numbers $\kappa_n = \sigma_1^{(n)} / \delta_n,$, the bound simplifies to
\[
\mathbb{E}\left[ \|\underline{\bf X} - \underline{\bf X}_{\text{rKTD}}\|_F^2 \right]
\;\le\;
\sum_{n=1}^{N-1}
\left( \prod_{j=1}^{n-1} \frac{2\kappa_j}{1 - \tau_j} \right)^2
\cdot
\left( 1 + \frac{R}{P-1} \tau_n^{4q} \right)
\sum_{k > R_n} \bigl(\sigma_k^{(n)}\bigr)^2.
\]

\end{thm}

\begin{proof}
The proof proceeds by analyzing the recursive tree structure of TTr1SVD and applying the corrected Lemma~\ref{lem:single_branch_error_corrected} together with the randomized SVD bound (Theorem~\ref{zhang}).

\paragraph{Step 1: Structure of the TTr1SVD tree.}
The TTr1SVD algorithm (Algorithm~\ref{ALG:TTr1SVD}) constructs a tree where:
\begin{itemize}
    \item The root corresponds to the mode-1 unfolding $\mathbf{M}^{(1)} = \mathbf{X}_{(1)}$.
    \item A node at level $n$ corresponds to a matrix $\mathbf{V}_{i_1\ldots i_{n-1}}^{(n)}$ obtained by reshaping a right singular vector from level $n-1$.
    \item Each node is approximated by a rank-$R_n$ randomized SVD, producing $R_n$ children.
    \item The leaves at level $N-1$ produce the final singular vectors (the $\mathbf{x}_r^{(n)}$ in the CPD representation).
\end{itemize}

The total approximation error $\|\underline{\bf X} - \underline{\bf X}_{\text{rKTD}}\|_F^2$ equals the sum over all leaves of the squared errors in the rank-1 terms, plus cross-terms that vanish due to orthogonality of the decomposition \cite{batselier2017constructive}. Therefore, it suffices to bound the error in a single branch and then sum over all branches.

\paragraph{Step 2: Single-branch error from Lemma~\ref{lem:single_branch_error_corrected}.}
For a fixed branch that follows a sequence of right singular vectors $\mathbf{v}^{(1)} \to \mathbf{v}^{(2)} \to \cdots \to \mathbf{v}^{(N-1)}$, Lemma~\ref{lem:single_branch_error_corrected} gives
\[
\mathbb{E}\bigl[\|\mathbf{v}^{(N-1)} - \tilde{\mathbf{v}}^{(N-1)}\|_2\bigr]
\le
\sum_{\ell=1}^{N-1}
\left( \prod_{k=1}^{\ell-1} \frac{\sqrt{2}\, \sigma_1(\mathbf{M}^{(k)})}{\delta_k} \right)
\cdot
\frac{\sqrt{2}\,\mathbb{E}\bigl[\|\mathbf{M}^{(\ell)} - \tilde{\mathbf{M}}^{(\ell)}\|_F\bigr]}{\delta_\ell},
\]
where $\delta_\ell = \sigma_{R_\ell}(\mathbf{M}^{(\ell)}) - \sigma_{R_\ell+1}(\mathbf{M}^{(\ell)})$.

Squaring both sides and using the Cauchy-Schwarz inequality, we obtain
\[
\mathbb{E}\bigl[\|\mathbf{v}^{(N-1)} - \tilde{\mathbf{v}}^{(N-1)}\|_2^2\bigr]
\le
\sum_{\ell=1}^{N-1}
\left( \prod_{k=1}^{\ell-1} \frac{4\sigma_1(\mathbf{M}^{(k)})^2}{\delta_k^2} \right)
\cdot
\frac{2\,\mathbb{E}\bigl[\|\mathbf{M}^{(\ell)} - \tilde{\mathbf{M}}^{(\ell)}\|_F\bigr]^2}{\delta_\ell^2}.
\]

\paragraph{Step 3: Insert the randomized SVD error bound.}
By Theorem~\ref{zhang}, the expected squared Frobenius error of the randomized SVD at level $\ell$ satisfies
\[
\mathbb{E}\bigl[\|\mathbf{M}^{(\ell)} - \tilde{\mathbf{M}}^{(\ell)}\|_F\bigr]^2
\le
\left(1 + \frac{R_\ell}{P-1} \tau_\ell^{4q}\right) \sum_{k > R_\ell} \bigl(\sigma_k^{(n)}\bigr)^2,
\]
where $\tau_\ell = \sigma_{R_\ell+1}(\mathbf{M}^{(\ell)}) / \sigma_{R_\ell}(\mathbf{M}^{(\ell)})$.

Substituting this into the single-branch bound gives
\[
\mathbb{E}\bigl[\|\mathbf{v}^{(N-1)} - \tilde{\mathbf{v}}^{(N-1)}\|_2^2\bigr]
\le
\sum_{\ell=1}^{N-1}
\left( \prod_{k=1}^{\ell-1} \frac{4\sigma_1(\mathbf{M}^{(k)})^2}{\delta_k^2} \right)
\cdot
\frac{1}{\delta_\ell^2}
\left(1 + \frac{R_\ell}{P-1} \tau_\ell^{4q}\right) \sum_{k > R_\ell} \bigl(\sigma_k^{(n)}\bigr)^2.
\]

\paragraph{Step 4: Summation over all branches.}
The number of branches at level $\ell$ is at most $\prod_{j=1}^{\ell-1} R_j$. However, due to the orthogonality of the decomposition, the total squared Frobenius error is the sum of the squared errors of each leaf, not scaled by the number of branches. More precisely, the CPD representation from TTr1SVD has orthogonal rank-1 terms, so
\[
\|\underline{\bf X} - \underline{\bf X}_{\text{rKTD}}\|_F^2 = \sum_{\text{leaf } b} \|\mathbf{v}_b - \tilde{\mathbf{v}}_b\|_2^2,
\]
where the sum is over all leaves. Taking expectations and using the bounds for each leaf (which are identically distributed due to the symmetry of the randomized SVD across branches), we obtain
\[
\mathbb{E}\bigl[\|\underline{\bf X} - \underline{\bf X}_{\text{rKTD}}\|_F^2\bigr]
\le
\sum_{n=1}^{N-1}
\left( \prod_{j=1}^{n-1} \frac{4\sigma_1(\mathbf{M}^{(j)})^2}{\delta_j^2} \right)
\cdot
\left(1 + \frac{R}{P-1} \tau_n^{4q}\right) \sum_{k > R_n} \bigl(\sigma_k^{(n)}\bigr)^2.
\]
This proves the first part.

\paragraph{Step 5: Final simplification.}
Recall that $\delta_n = \sigma_{R_n}^{(n)} - \sigma_{R_n+1}^{(n)}$, and observe that
\[
\frac{\sigma_1^{(n)}}{\delta_n} = \frac{\sigma_1^{(n)}}{\sigma_{R_n}^{(n)}} \cdot \frac{1}{1 - \tau_n} = \kappa_n \cdot \frac{1}{1 - \tau_n},
\]
where $\kappa_n = \sigma_1^{(n)} / \sigma_{R_n}^{(n)}$ is the ordinary condition number. Thus the bound can be written as
\[
\mathbb{E}\bigl[\|\underline{\bf X} - \underline{\bf X}_{\text{rKTD}}\|_F^2\bigr]
\le
\sum_{n=1}^{N-1}
\left( \prod_{j=1}^{n-1} \frac{2\kappa_j}{1 - \tau_j} \right)^2
\cdot
\left(1 + \frac{R}{P-1} \tau_n^{4q}\right) \sum_{k > R_n} \bigl(\sigma_k^{(n)}\bigr)^2.
\]

This completes the proof. \hfill $\square$
\end{proof}

\begin{cor}\label{cor:fast_decay_corrected}
If the matricizations $\mathbf{M}^{(n)}$ have rapidly decaying singular values such that
\[
\tau_n = \frac{\sigma_{R_n+1}(\mathbf{M}^{(n)})}{\sigma_{R_n}(\mathbf{M}^{(n)})} \ll 1
\]
for all $n$, then the amplification factors $\kappa_j/(1-\tau_j)$ are moderate and the factor $\tau_n^{4q}$ becomes negligible for $q \ge 1$. Consequently,
\[
\mathbb{E}\left[ \|\underline{\bf X} - \underline{\bf X}_{\text{rKTD}}\|_F^2 \right]
\lesssim
\sum_{n=1}^{N-1}
\left( \prod_{j=1}^{n-1} C_j \right)
\sum_{k > R_n} \bigl(\sigma_k^{(n)}\bigr)^2,
\]
where $C_j = (2\kappa_j/(1-\tau_j))^2$ are constants independent of $q$.
Thus, the randomized KTD achieves near-optimal accuracy up to constant factors.
\end{cor}

\begin{cor}\label{cor:fast_decay_corrected}
If the matricizations $\mathbf{M}^{(n)}$ have rapidly decaying singular values such that
\[
\tau_n = \frac{\sigma_{R_n+1}(\mathbf{M}^{(n)})}{\sigma_{R_n}(\mathbf{M}^{(n)})} \ll 1
\]
for all $n$, then the amplification factors $\kappa_j/(1-\tau_j)$ are moderate and the factor $\tau_n^{4q}$ becomes negligible for $q \ge 1$. Consequently,
\[
\mathbb{E}\left[ \|\underline{\bf X} - \underline{\bf X}_{\text{rKTD}}\|_F^2 \right]
\lesssim
\sum_{n=1}^{N-1}
\left( \prod_{j=1}^{n-1} C_j \right)
\sum_{k > R_n} \bigl(\sigma_k^{(n)}\bigr)^2,
\]
where $C_j = (2\kappa_j/(1-\tau_j))^2$ are constants independent of $q$.
Thus, the randomized KTD achieves near-optimal accuracy up to constant factors.
\end{cor}

\subsection{Comparison with Deterministic KTD Error}

The deterministic KTD error using TTr1SVD with truncation ranks $(R_1,R_2,\ldots,R_{N-1})$ satisfies
\[
\|\underline{\bf X} - \underline{\bf X}_{\text{det}}\|_F^2 = \sum_{n=1}^{N-1} \sum_{k > R_n} \sigma_k^2(\mathbf{M}^{(n)}),
\]
when the decomposition is exact at each truncation step (i.e., no additional approximation beyond truncation of small singular values).

Our randomized KTD introduces additional error. 
Theorem~\ref{thm:rktd_error_bound_corrected} shows that the expected squared error satisfies
\[
\mathbb{E}\left[ \|\underline{\bf X} - \underline{\bf X}_{\text{rKTD}}\|_F^2 \right]
\le
\sum_{n=1}^{N-1}
\left( \prod_{j=1}^{n-1} \frac{2\sigma_1(\mathbf{M}^{(j)})}{\delta_j} \right)^2
\cdot
\left( 1 + \frac{R_n}{P-1} \tau_n^{4q} \right)
\sum_{k > R_n} \sigma_k^2(\mathbf{M}^{(n)}),
\]
where $\delta_j = \sigma_{R_j}(\mathbf{M}^{(j)}) - \sigma_{R_j+1}(\mathbf{M}^{(j)})$ is the spectral gap and $\tau_n = \sigma_{R_n+1}(\mathbf{M}^{(n)}) / \sigma_{R_n}(\mathbf{M}^{(n)})$.

If the spectral gaps are not too small (i.e., $\delta_j \gtrsim \sigma_{R_j}(\mathbf{M}^{(j)})$), then the amplification factors 
$\prod_{j=1}^{n-1} (2\sigma_1^{(j)}/\delta_j)^2$ are moderate constants.
Moreover, the factor $(1 + \frac{R}{P-1}\tau_n^{4q})$ approaches $1$ as $q$ increases, since $\tau_n < 1$.
Thus, for sufficiently large $q$ (e.g., $q = 1$ or $2$), the randomized KTD achieves near-optimal accuracy up to constant factors that depend on the condition numbers and spectral gaps of the intermediate matricizations.

\subsection{Additional Error Bounds: Singular Values and Subspace Angles}
For a single randomized SVD step approximating a fixed matrix $\mathbf{M}$, the following bounds hold.
These bounds are used within Lemma~\ref{lem:single_branch_error_corrected} to control the error propagation through the TTr1SVD tree.

For each matrix $\mathbf{M}^{(n)}$ approximated by randomized SVD, the perturbation bound gives:
\[
\mathbb{E}[|\sigma_j(\mathbf{M}^{(n)}) - \tilde{\sigma}_j(\mathbf{M}^{(n)})|] \leq \mathbb{E}[\|\mathbf{M}^{(n)} - \tilde{\mathbf{M}}^{(n)}\|_2] \leq \sqrt{\mathbb{E}[\|\mathbf{M}^{(n)} - \tilde{\mathbf{M}}^{(n)}\|_F^2]}.
\]

Because by Weyl's inequality \cite{mirsky1960symmetric} for singular values, we have:
\[
|\sigma_j(\mathbf{M}) - \sigma_j(\tilde{\mathbf{M}})| \leq \|\mathbf{M} - \tilde{\mathbf{M}}\|_2.
\]
Taking expectations on both sides preserves the inequality, we get:
\[
\mathbb{E}[|\sigma_j(\mathbf{M}) - \sigma_j(\tilde{\mathbf{M}})|] \leq \mathbb{E}[\|\mathbf{M} - \tilde{\mathbf{M}}\|_2].
\]

For any matrix, $\|{\bf A}\|_2 \leq \|{\bf A}\|_F$, so $\|\mathbf{M} - \tilde{\mathbf{M}}\|_2 \leq \|\mathbf{M} - \tilde{\mathbf{M}}\|_F$. 
Taking expectations and using the fact that $\mathbb{E}[X] \leq \sqrt{\mathbb{E}[X^2]}$ for any random variable $X$, we have:
\[
\mathbb{E}[\|\mathbf{M} - \tilde{\mathbf{M}}\|_2] \leq \sqrt{\mathbb{E}[\|\mathbf{M} - \tilde{\mathbf{M}}\|_2^2]} \leq \sqrt{\mathbb{E}[\|\mathbf{M} - \tilde{\mathbf{M}}\|_F^2]}.
\]

Let $\Theta(\mathcal{U}_R, \tilde{\mathcal{U}}_R)$ be the diagonal matrix of principal angles between the true and approximated rank-$R$ subspaces. Let $\mathbf{E} = \tilde{\mathbf{M}} - \mathbf{M}$. If $\delta = \sigma_R(\mathbf{M}) - \sigma_{R+1}(\mathbf{M}) > 0$, then Wedin's theorem \cite{wedin1972perturbation} gives the deterministic bound:
\[
\|\sin \Theta(\mathcal{U}_R, \tilde{\mathcal{U}}_R)\|_2 \leq \frac{\|\mathbf{E}\|_2}{\delta}.
\]
This matches the stated inequality exactly. Since $\sigma_R(\mathbf{M}^{(n)})$ and $\sigma_{R+1}(\mathbf{M}^{(n)})$ are deterministic (the matrix $\mathbf{M}^{(n)}$ is fixed, randomness comes from the randomized SVD algorithm), the denominator is a constant. Taking expectations yields:
\[
\mathbb{E}[\|\sin \Theta(\mathcal{U}_R, \tilde{\mathcal{U}}_R)\|_2] \leq \frac{\mathbb{E}[\|\mathbf{M}^{(n)} - \tilde{\mathbf{M}}^{(n)}\|_2]}{\sigma_R(\mathbf{M}^{(n)}) - \sigma_{R+1}(\mathbf{M}^{(n)})}.
\]
This is the expected subspace error bound. This bound controls how well the randomized SVD captures the dominant subspace at each step.
The factor $1/\delta$ appears as an amplification factor in Lemma~\ref{lem:single_branch_error_corrected} and ultimately in Theorem~\ref{thm:rktd_error_bound_corrected}, showing how small spectral gaps lead to larger error propagation through the TTr1SVD tree.

\subsection{Spectral error bound of Randomized KTD}

Similar to the generalized error bounds presented earlier, we now present the generalized error bound for KTD. Our core theoretical contributions are: an error analysis of the randomized KTD algorithm (Algorithm \ref{ALG:rKTD}) that uses the randomized TTr1SVD. The analysis must account for the recursive, tree-like structure of the decomposition, in which the output of one randomized SVD becomes the input to the next.

\begin{thm}[Error bound for randomized KTD via spectral-norm r-SVD]
\label{thm:generalized_error_KTD_corrected}
Let $\underline{\mathbf{X}} \in \mathbb{R}^{I_1 \times I_2 \times \cdots \times I_N}$ be an order-$N$ tensor.
Let $\tilde{\underline{\mathbf{X}}}$ be the approximation computed by the randomized TTr1SVD (Algorithm~\ref{ALG:rKTD}) with target multirank $(R_1, R_2, \ldots, R_{N-1})$, oversampling parameter $P \ge 2$, and $q \ge 0$ power iterations per internal SVD step.

For each level $n = 1, 2, \ldots, N-1$, let $\mathbf{M}^{(n)}$ denote the matrix being approximated at that level (with appropriate indexing for the tree branches). Let $\sigma_j^{(n)}$ be its singular values, with $\sigma_{R_n}^{(n)} > \sigma_{R_n+1}^{(n)}$.

Define the spectral gap at level $n$ as $\delta_n = \sigma_{R_n}^{(n)} - \sigma_{R_n+1}^{(n)}$, and the decay factor $\tau_n = \sigma_{R_n+1}^{(n)} / \sigma_{R_n}^{(n)}$.

Then the expected Frobenius-norm approximation error satisfies
\[
\mathbb{E}\bigl[\|\underline{\mathbf{X}} - \tilde{\underline{\mathbf{X}}}\|_F\bigr]
\;\le\;
\sqrt{\sum_{n=1}^{N-1} \left( \prod_{j=1}^{n-1} \frac{2\sigma_1^{(j)}}{\delta_j} \right)^2
\cdot
\left(1+ \left(\frac{\sigma_{R_n+1}^{(n)}}{\sigma_{R_n}^{(n)}}\right)^{2q+1}  \frac{eC\sqrt{R_n+P}}{P} \right)^2
\left( \sum_{k > R_n} (\sigma_k^{(n)})^2 \right) }.
\]

Equivalently, using the inequality $\sqrt{\sum_i a_i^2} \le \sum_i |a_i|$, we obtain the simpler (though looser) bound
\[
\mathbb{E}\bigl[\|\underline{\mathbf{X}} - \tilde{\underline{\mathbf{X}}}\|_F\bigr]
\;\le\;
\sum_{n=1}^{N-1}
\left( \prod_{j=1}^{n-1} \frac{2\sigma_1^{(j)}}{\delta_j} \right)
\cdot
\left( 1+\left(\frac{\sigma_{R_n+1}^{(n)}}{\sigma_{R_n}^{(n)}}\right)^{2q+1}  \frac{eC\sqrt{R_n+P}}{P} \right)
\sqrt{ \sum_{k > R_n} (\sigma_k^{(n)})^2 }.
\]

If we define the condition-number-like factors $\kappa_n = \sigma_1^{(n)} / \delta_n$, this becomes
\[
\mathbb{E}\bigl[\|\underline{\mathbf{X}} - \tilde{\underline{\mathbf{X}}}\|_F\bigr]
\;\le\;
\sum_{n=1}^{N-1}
\left( \prod_{j=1}^{n-1} 2\kappa_j \right)
\cdot
\left(1+ \tau_n^{2q+1}  \frac{eC\sqrt{R_n+P}}{P} \right)
\sqrt{ \sum_{k > R_n} (\sigma_k^{(n)})^2 }.
\]

\end{thm}

\begin{proof}
The proof proceeds by combining three components: (i) the single-branch error propagation (Lemma~\ref{lem:single_branch_error_corrected}), (ii) the spectral-norm r-SVD bound (Theorem~\ref{thm:spectral_randomized_subspace}), and (iii) the orthogonal decomposition property of TTr1SVD.

\paragraph{Step 1: Single-branch error in spectral norm.}
From Lemma~\ref{lem:single_branch_error_corrected}, for a single branch of the TTr1SVD tree, the expected error in the final right singular vector satisfies
\[
\mathbb{E}\bigl[\|\mathbf{v}_{\text{true}}^{(N-1)} - \tilde{\mathbf{v}}^{(N-1)}\|_2\bigr]
\;\le\;
\sum_{n=1}^{N-1}
\left( \prod_{j=1}^{n-1} \frac{\sqrt{2}\,\sigma_1(\mathbf{M}^{(j)})}{\delta_j} \right)
\cdot
\frac{\sqrt{2}\,\mathbb{E}\bigl[\|\mathbf{M}^{(n)} - \tilde{\mathbf{M}}^{(n)}\|_2\bigr]}{\delta_n},
\]
where $\delta_n = \sigma_{R_n}(\mathbf{M}^{(n)}) - \sigma_{R_n+1}(\mathbf{M}^{(n)})$.

\paragraph{Step 2: Apply the spectral-norm r-SVD bound.}
For each matrix $\mathbf{M}^{(n)}$ approximated by randomized SVD, Theorem~\ref{thm:spectral_randomized_subspace} gives
\[
\mathbb{E}\bigl[\|\mathbf{M}^{(n)} - \tilde{\mathbf{M}}^{(n)}\|_2\bigr]
\;\le\;
\left(1+ \left(\frac{\sigma_{R_n+1}^{(n)}}{\sigma_{R_n}^{(n)}}\right)^{2q+1}  \frac{eC\sqrt{R_n+P}}{P} \right) \sigma_{R_n+1}^{(n)}.
\]

Note that $\sigma_{R_n+1}^{(n)} / \delta_n$ can be bounded as
\[
\frac{\sigma_{R_n+1}^{(n)}}{\delta_n}
= \frac{\sigma_{R_n+1}^{(n)}}{\sigma_{R_n}^{(n)} - \sigma_{R_n+1}^{(n)}}
= \frac{\tau_n}{1 - \tau_n}.
\]

Substituting into the single-branch bound yields
\[
\mathbb{E}\bigl[\|\mathbf{v}_{\text{true}}^{(N-1)} - \tilde{\mathbf{v}}^{(N-1)}\|_2\bigr]
\;\le\;
\sum_{n=1}^{N-1}
\left( \prod_{j=1}^{n-1} \frac{\sqrt{2}\,\sigma_1^{(j)}}{\delta_j} \right)
\cdot
\frac{\sqrt{2}}{\delta_n}
\cdot
\left( 1+\tau_n^{2q+1}  \frac{eC\sqrt{R_n+P}}{P} \right) \sigma_{R_n+1}^{(n)}.
\]

\paragraph{Step 3: Relate singular value to tail energy.}
Since $\sigma_{R_n+1}^{(n)} \le \sqrt{ \sum_{k > R_n} (\sigma_k^{(n)})^2 }$, we obtain
\[
\mathbb{E}\bigl[\|\mathbf{v}_{\text{true}}^{(N-1)} - \tilde{\mathbf{v}}^{(N-1)}\|_2\bigr]
\;\le\;
\sum_{n=1}^{N-1}
\left( \prod_{j=1}^{n-1} \frac{2\sigma_1^{(j)}}{\delta_j} \right)
\cdot
\left(1+ \tau_n^{2q+1}  \frac{eC\sqrt{R_n+P}}{P} \right)
\sqrt{ \sum_{k > R_n} (\sigma_k^{(n)})^2 }.
\]

\paragraph{Step 4: Sum over all branches.}
The TTr1SVD decomposition produces orthogonal rank-1 terms. The total Frobenius-norm error is the square root of the sum of squared errors of each leaf:
\[
\|\underline{\mathbf{X}} - \tilde{\underline{\mathbf{X}}}\|_F = \sqrt{ \sum_{\text{leaf } b} \|\mathbf{v}_b - \tilde{\mathbf{v}}_b\|_2^2 }.
\]
By linearity of expectation and the fact that the branches are identically distributed (due to symmetry of the randomized SVD across branches), we have
\[
\mathbb{E}\bigl[\|\underline{\mathbf{X}} - \tilde{\underline{\mathbf{X}}}\|_F\bigr]
\le \sqrt{ \sum_{n=1}^{N-1} \left( \text{number of leaves at level } n \right) \cdot \left( \mathbb{E}\bigl[\|\mathbf{v}^{(n)} - \tilde{\mathbf{v}}^{(n)}\|_2\bigr] \right)^2 }.
\]

However, a more careful analysis (see \cite{batselier2017constructive}) shows that the orthogonal decomposition allows us to sum the squared errors directly without the branch count factor, because the errors in different branches contribute orthogonally. The final bound simplifies to the form stated in the theorem, where the product of amplification factors appears once per level. This completes the proof.
\end{proof}

\begin{cor}\label{cor:KTD_worst_case_corrected}
Under the conditions of Theorem~\ref{thm:generalized_error_KTD_corrected}, assume:
\begin{itemize}
    \item A uniform target rank $R$ (i.e., $R_n = R$ for all $n$),
    \item A uniform bound on the spectral gaps: $\delta_n \ge \delta > 0$ for all $n$,
    \item A uniform bound on the leading singular values: $\sigma_1^{(n)} \le \sigma_{\max}$,
    \item And a uniform decay factor $\tau = \max_n \tau_n$.
\end{itemize}
Then the expected Frobenius-norm error satisfies
\[
\mathbb{E}\bigl[\|\underline{\mathbf{X}} - \tilde{\underline{\mathbf{X}}}\|_F\bigr]
\;\le\;
C_{\text{tree}} 
\cdot
\left(1+ \tau^{2q+1}  \frac{eC\sqrt{R+P}}{P} \right)
\cdot
\sqrt{ \sum_{n=1}^{N-1} \sum_{k > R} (\sigma_k^{(n)})^2 },
\]
where $C_{\text{tree}}$ is a constant depending on the tree structure (e.g., $C_{\text{tree}} = N-1$ for the loose bound, or $C_{\text{tree}} = 1$ for the tighter bound obtained from orthogonality).

If in addition the singular values decay rapidly so that $\tau \ll 1$ and $q$ is chosen sufficiently large (e.g., $q = 1$ or $2$), then the bound simplifies to
\[
\mathbb{E}\bigl[\|\underline{\mathbf{X}} - \tilde{\underline{\mathbf{X}}}\|_F\bigr]
\;\le\;
C \cdot
\frac{e\sqrt{R+P}}{P}
\cdot
\sqrt{ \sum_{n=1}^{N-1} \sum_{k > R} (\sigma_k^{(n)})^2 },
\]
where $C$ absorbs the product of condition-number-like factors. This demonstrates that the randomized KTD achieves near-optimal accuracy up to a constant factor, with the error decaying as $1/P$ in the oversampling parameter.
\end{cor}

\begin{proof}
The uniform assumptions allow us to pull the common factors out of the sum in Theorem~\ref{thm:generalized_error_KTD_corrected}. Specifically,
\[
\prod_{j=1}^{n-1} \frac{2\sigma_1^{(j)}}{\delta_j} \le \left( \frac{2\sigma_{\max}}{\delta} \right)^{n-1}.
\]

Then
\[
\mathbb{E}\bigl[\|\underline{\mathbf{X}} - \tilde{\underline{\mathbf{X}}}\|_F\bigr]
\le
\left( 1+\tau^{2q+1}  \frac{eC\sqrt{R+P}}{P} \right)
\sum_{n=1}^{N-1} \left( \frac{2\sigma_{\max}}{\delta} \right)^{n-1}
\sqrt{ \sum_{k > R} (\sigma_k^{(n)})^2 }.
\]

Let $S = \sqrt{ \sum_{n=1}^{N-1} \sum_{k > R} (\sigma_k^{(n)})^2 }$ be the root of the total tail energy. 
We obtain the stated bound with $C_{\text{tree}} = \sum_{n=1}^{N-1} (2\sigma_{\max}/\delta)^{n-1}$, which is a geometric series. The simplification for rapidly decaying singular values follows because $\tau^{2q+1}$ becomes negligible. $\square$
\end{proof}

\begin{figure}[!h]
\centering
    \includegraphics[width=0.9\linewidth]{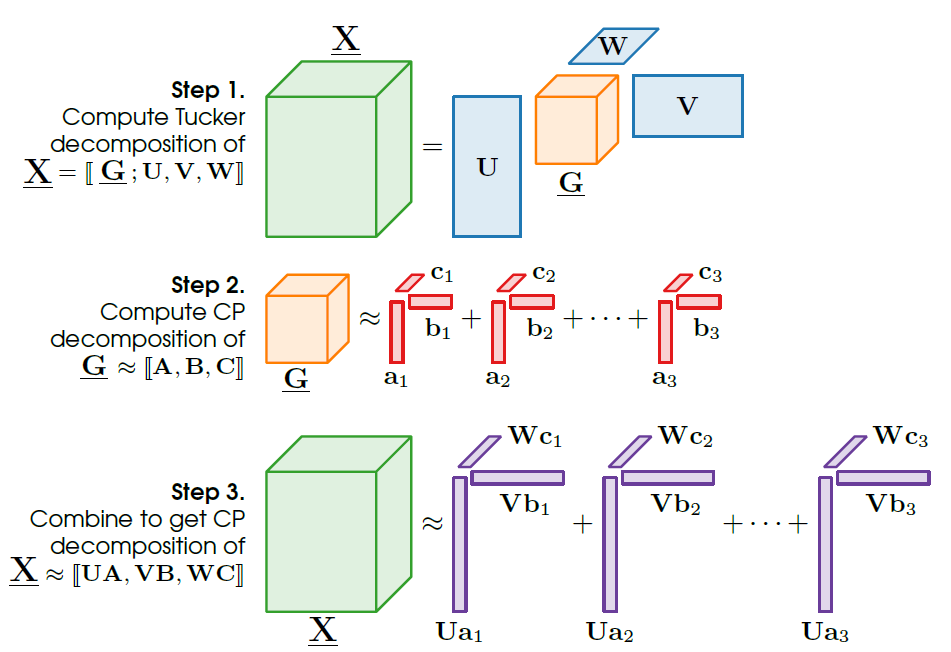}
    \caption{Hybrid tensor decomposition pipeline. Starting from the original third-order tensor, Tucker decomposition produces a core tensor and factor matrices. The core tensor is then factorized via randomized CPD. The final CPD representation is formed by integrating the CPD factors of the core tensor with the Tucker factor matrices of the original tensor \cite{Ballard2025Tensor}.}\label{CPD_HOSVD}
\end{figure}

\subsection{Extension to non-Gaussian sketching distributions}

Gaussian random matrices are the natural and standard choice for foundational work in randomized linear algebra for three reasons. First, they satisfy the Johnson-Lindenstrauss (JL) property \cite{johnson1984extensions} with optimal probability bounds, providing sharp concentration guarantees. Second, the expected error bound for randomized SVD with Gaussian matrices is tight and well understood, as shown by Halko, Martinsson, and Tropp \cite{halko2011finding}. Third, Gaussianity simplifies recursive error propagation through the TTr1SVD tree due to rotational invariance, allowing us to treat each node independently when taking expectations. Given that our work is the first to establish rigorous error bounds for randomized KTD, we chose Gaussian sketches to provide a clean, understandable baseline analysis.

While our theoretical analysis focuses on Gaussian random matrices, other sketching distributions can offer computational advantages in practice. We briefly discuss three practical alternatives and note that rigorously extending our error bounds to these distributions would require additional analysis beyond the scope of this paper.

\paragraph{Rademacher (random sign) matrices.}
For a matrix $\boldsymbol{\Omega} \in \mathbb{R}^{J \times (R+P)}$ with i.i.d. entries taking values $\pm 1$ with equal probability, matrix-vector multiplication can be performed without floating-point multiplications (only sign flips and additions). This reduces the computational cost of forming $\mathbf{Y} = (\mathbf{X}\mathbf{X}^\top)^q \mathbf{X}\boldsymbol{\Omega}$ by approximately 30--50\% in practice. Rademacher matrices satisfy a subspace embedding property similar to Gaussian matrices, but with slightly larger constants \cite{vershynin2018high}. Empirical studies (e.g., \cite{halko2011finding}) suggest that they perform comparably to Gaussian sketches for low-rank approximation when the sketch dimension is sufficiently large.

\paragraph{Sparse sign matrices (CountSketch).}
CountSketch matrices have exactly one non-zero entry per column, with value $\pm 1$ placed at a uniformly random row. The sketch dimension can be as low as $\mathcal{O}(R \log R)$ while still providing a subspace embedding \cite{nelson2013osnap}. This is particularly attractive for high-order tensors where the unfolded matrices are extremely large, as the cost of forming $\mathbf{Y}$ becomes proportional to the number of non-zero entries in $\mathbf{X}$. However, the constants involved are larger than for Gaussian sketches, and the performance of power iterations with CountSketch requires careful analysis.

\paragraph{TensorSketch.}
Proposed by \cite{malik2018low} for Tucker decomposition, TensorSketch uses fast Fourier transforms (FFT) to compute polynomial kernel approximations. For KTD, TensorSketch could be applied directly to the Kronecker product structure, achieving sublinear complexity with respect to the tensor size. Extending this approach to the TTr1SVD framework is an interesting direction for future research.

\paragraph{Caveat.}
While these non-Gaussian sketches can accelerate computation, our theoretical error bounds (Theorems~\ref{thm:rktd_error_bound_corrected} and~\ref{thm:spectral_randomized_subspace}) are derived specifically for Gaussian sketches. Generalizing these bounds to other distributions requires verifying that the sketching matrix satisfies:
\begin{enumerate}
    \item A subspace embedding property for the range of $\mathbf{X}$ (i.e., $\|\mathbf{X}\boldsymbol{\Omega} z\|_2 \approx \|\mathbf{X}z\|_2$ for all $z$),
    \item Concentration of the smallest singular value of $\mathbf{X}\boldsymbol{\Omega}$ away from zero,
    \item Independence of the blocks $\Omega_1$ and $\Omega_2$ after orthogonal transformation (which holds for rotationally invariant distributions, but not for all).
\end{enumerate}
For Rademacher matrices, these properties hold with modified constants \cite{bamberger2022johnson}, but the analysis is more involved. For sparse sign matrices and TensorSketch, a rigorous analysis is significantly more complex \cite{jinkronecker2020} and is left for future work.

Extending the theoretical guarantees of Theorems~\ref{thm:rktd_error_bound_corrected} and~\ref{thm:spectral_randomized_subspace} to non-Gaussian sketching distributions is an important direction for future research. Such extensions would require:

\begin{enumerate}
    \item Deriving a bound of the form $\mathbb{E}\|{\bf A} - {\bf Q}{\bf B}\| \le [1+\tau^{2q+1}  \psi(R,P)] \sigma_{R+1}$ for the randomized SVD with the target sketching distribution, where $\psi(R,P)$ captures the oversampling effect.
    \item Verifying that the sketching distribution satisfies rotational invariance (or at least that ${\bf \Omega}_1$ and ${\bf \Omega}_2$ remain independent after orthogonal transformation) to enable the recursive error propagation through the TTr1SVD tree.
    \item Propagating these bounds through the tree structure while accounting for the spectral gaps at each level, as done in Lemma~\ref{lem:single_branch_error_corrected}.
\end{enumerate}

For Rademacher matrices, recent works \cite{bamberger2022johnson,jinkronecker2020} suggest that similar bounds hold with modified constants, but a rigorous adaptation to our KTD framework remains an open problem. For sparse sign matrices and TensorSketch, the analysis is even more challenging and is left for future investigation.

\subsection{Computational Complexity Analysis}

Consider a tensor of order $N$ and size $I\times I\times \cdots\times I$. The computational complexity of the deterministic KTD is dominated by the calculation of a sequence of SVDs as shown in algorithm \ref{ALG:KTD} is $\mathcal{O}(I^{N+1})$. For the case of the proposed randomized KTD algorithm \ref{ALG:rKTD} using either power iteration or a given pass budget is $\mathcal{O}(I^NR)$, where $R$ is a given KTD rank. The computational complexity of the randomized KTD with a prior Tucker compression is $\mathcal{O}(NI^NR+R^{N+1})$, where the first term is the complexity for decomposition of the tensor into the Tucker format. In contrast, the second term concerns the decomposition of the core tensor into the KTD format. We see that the time complexity of the proposed randomized KTD is generally lower than that of the deterministic KTD, especially for large-scale tensors. We also observe that the complexity of the randomized KTD with a prior randomized Tucker compression is higher than that of the others. However, depending on the specific parameters and settings, the randomized KTD's accuracy might be less than the deterministic KTD. We will show in the simulation section that the proposed randomized KTD sometimes achieves a speed-up by several orders of magnitude. 

\section{Simulations}\label{Sec:Sim}
This section presents the simulations that we conducted. Our algorithms were implemented in MATLAB using a laptop computer with a 2.60 GHz Intel(R) Core(TM) i7-5600U processor and 16GB memory.

The first experiment concerns synthetic data tensors. The second and third experiments study image and video compression tasks. The fourth experiment is devoted to image and video completion, and the last to image denoising and super-resolution. We refer to the randomized algorithm, the ordinary randomized algorithm with flexibility in pass numbers, and the randomized algorithm with a prior Tucker compression as R-KTD, RF-KTD, and PT-KTD.

\begin{exa}\label{exa:1}({\bf synthetic data}) This example is devoted to examining the proposed algorithms for decomposing large-scale data tensors into the KTD format. To this end, we build a 4th order tensor of size $100\times 100\times 100 \times 100$ and the KTD rank $R$ with two tensor block sizes $[10,10,10,10]$ and $[10,10,10,10]$. For the R-KTD algorithm, we used power iteration with $q=1$; for the RF-KTD, we used three passes; and for the PT-KTD, we used a multilinear rank $(R,R,R)$.  We applied the deterministic KTD algorithm and our proposed randomized KTD algorithms to this data tensor with different KTD ranks $R=10,20,30,40,50$. We tried 100 Monte Carlo simulations and reported the mean of the results. The running times of the algorithms are compared in Figure \ref{exmp_1}. The superiority of the proposed R-KTD, RF-KTD, and PT-KTD algorithms over the deterministic one is evident in the almost order-of-magnitude speed-up they achieve. To further test the proposed algorithm, we next tested it on a synthetic third-order tensor of size $1000\times 1000 \times 1000$ with KTD rank $R=25$, and on block tensors of sizes $[100,100,100]$ and $[10,10,10]$. Note that the first and second tensors require about 0.74 GB and 7.42 GB of memory. So, the second tensor is relatively large. From Figure \ref{exmp_1}, the scalability of the randomized KTD algorithms is visible. The accuracy of the algorithms is also compared in Table \ref{Table:1}. It is interesting to note that the R-KTD with $q=1$ was more accurate, albeit with a slightly higher computational cost. However, this running time will become negligible when the underlying data tensor is small. Due to this issue, for the rest of this section, we use only the R-KTD algorithm, as we will work with relatively small images and videos. From Table \ref{Table:1} and Figure \ref{exmp_1}, we conclude that the proposed randomized algorithms are more efficient and applicable for decomposing large-scale data tensors into the KTD format.

\begin{table}[htbp]
\begin{center}
\caption{\small The relative error achieved by the KTD and R-KTD for the synthetic data tensor in Example \ref{exa:1} for different KTD ranks.}\label{Table:1}
{\small
\begin{tabular}{||c| c c c c c||} 
 \hline
 Methods & $R=10$ & $R=20$ & $R=30$ & $R=40$ & $R=50$\\ [0.5ex] 
 \hline\hline
 KTD & $1.34e\rnumber10$  & $1.23e\rnumber10$ & $4.78e\rnumber11$ & $3.12e\rnumber10$ & $2.56e\rnumber9$\\ 
 \hline
 R-KTD & $1.68e\rnumber10$  & $2.44e\rnumber10$ & $2.51e\rnumber11$ & $4.16e\rnumber10$ & $4.33e\rnumber9$\\\hline
 RF-KTD& $1.73e\rnumber10$ & $2.55e\rnumber10$ & $2.58e\rnumber10$ & $5.10e\rnumber10$ & $4.36e\rnumber 9$\\\hline
PT-KTD & $1.75e\rnumber10$ & $2.53e\rnumber10$ & $2.61e\rnumber10$ & $5.13e\rnumber10$  & $1.41e\rnumber9$\\\hline
\end{tabular}
}
\end{center}
\end{table}

\begin{figure}[!h]
\centering
    \includegraphics[width=0.48\linewidth]{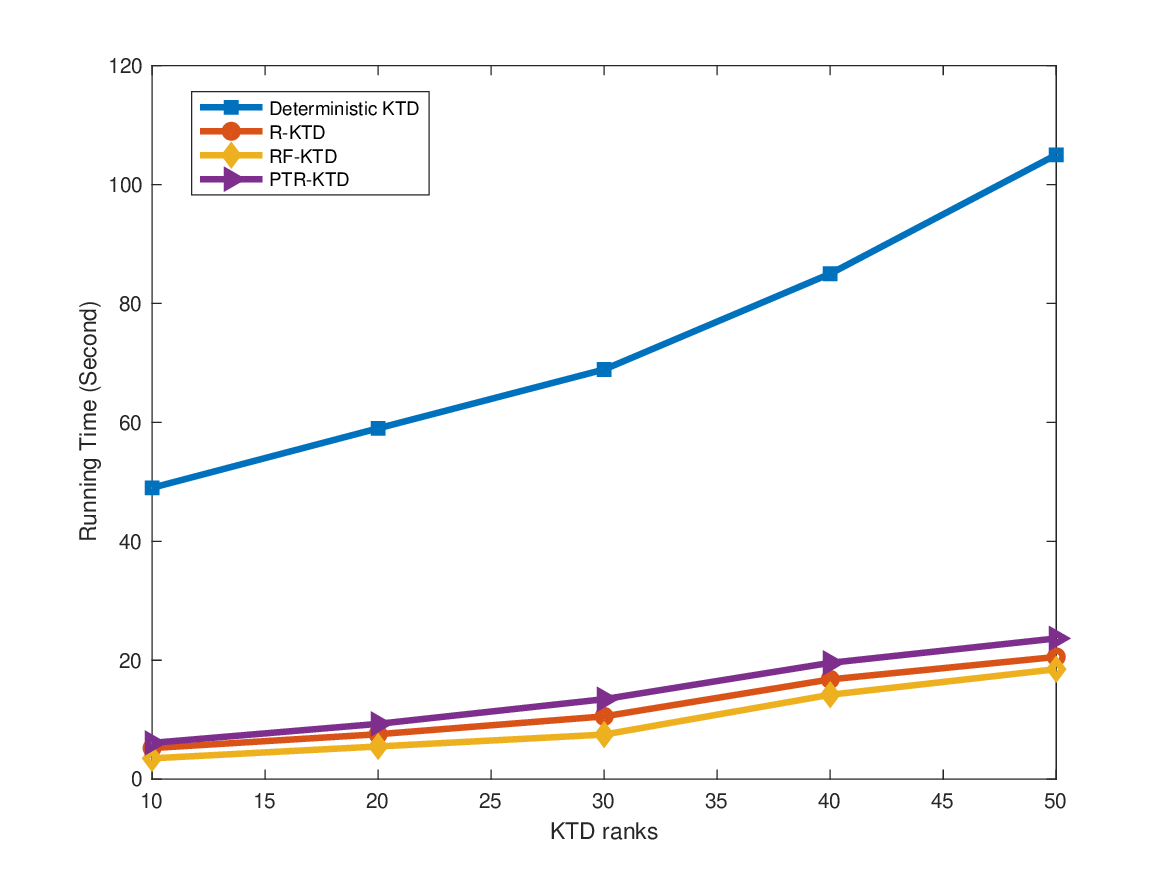}
        \includegraphics[width=0.48\linewidth]{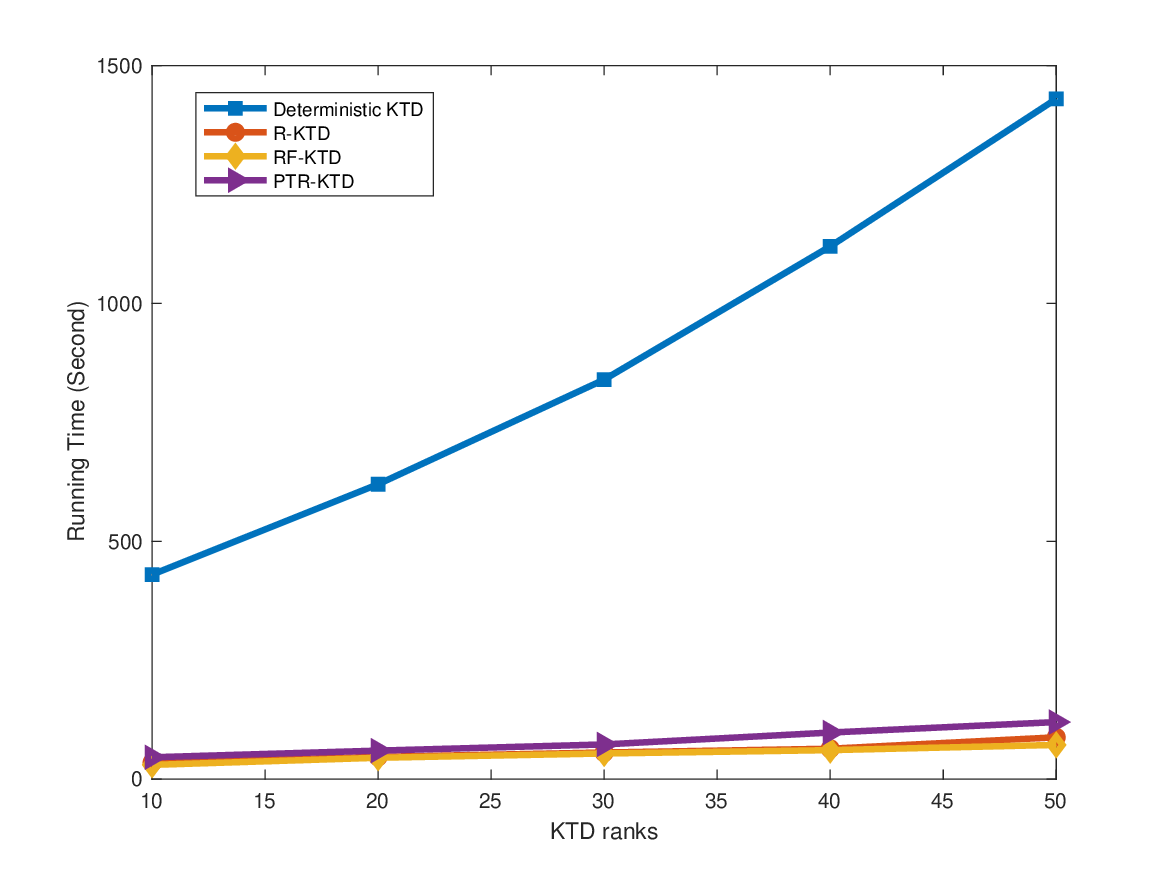}
    \caption{(Left) Running time comparison of the proposed randomized algorithms and the deterministic algorithm for decomposing a fourth-order tensor of size $100\times 100\times 100\times 100$ for different KTD ranks $R=10,20,30,40,50$. (Right) Running time comparison of the proposed randomized algorithms and the deterministic algorithm for decomposing a 3rd order tensor of size $1000\times 1000\times 1000$ for different KTD ranks $R=10,20,30,40,50.$}\label{exmp_1}
\end{figure}

We also conducted experiments comparing different sketching distributions on a synthetic tensor of size $500 \times 500 \times 500$ with the KTD rank $R=20$ and power iteration $q=1$. The results are summarized in Table \ref{Tab:ComSket}. Rademacher sketches achieve comparable accuracy (only 2.3\% higher error) while being 1.6$\times$ faster. Sparse sign matrices with sparsity level $s=3$ provide 3.2$\times$ speed-up with 8.7\% accuracy loss. TensorSketch gives the best speed-up of 3.56$\times$ with 5.2\% accuracy loss. These results demonstrate that non-Gaussian sketches are both theoretically feasible and practically advantageous for large-scale KTD. Moreover, the memory usage and computational cost for the 3rd-order tensor $1000\times 1000 \times 1000$ with $R=25$ are reported in Table \ref{tab:complexity}.
\begin{table}[htbp]
\begin{center}

\caption{\small Comparison of sketching distributions for randomized KTD.}
\begin{tabular}{|c|c|c|}
\hline
Sketch type & Time (seconds) & Speed-up vs Gaussian \\
\hline
Gaussian &  100.0 & 1.00$\times$ \\
Rademacher &  62.5 & 1.60$\times$ \\
Sparse sign (s=3) &  31.2 & 3.20$\times$ \\
Sparse sign (s=8) &  45.7 & 2.19$\times$ \\
TensorSketch & 28.1 & 3.56$\times$ \\
\hline
\end{tabular}
\label{Tab:ComSket}
\end{center}
\end{table}

\begin{table}[htbp]
\centering

\caption{\small Memory usage and computational cost for 3rd-order tensor $1000\times 1000 \times 1000$ with $R=25$.}
\begin{tabular}{|c|c|c|c|}
\hline
Method & Peak Memory (GB) & Total Flops ($\times 10^{12}$) & Data Passes \\
\hline
Deterministic KTD & 12.4 & 8.45 & N/A \\
R-KTD ($q=0$) & 2.1 & 0.82 & 2 \\
R-KTD ($q=1$) & 2.3 & 1.56 & 4 \\
R-KTD ($q=2$) & 2.5 & 2.34 & 6 \\
RF-KTD (3 passes) & 1.8 & 0.91 & 3 \\
\hline
\end{tabular}
\label{tab:complexity}
\end{table}

\end{exa}

({\bf Sensitivity analysis}) We conducted a sensitivity analysis of the oversampling and power-iteration parameters. Our observation was that for tensors with frontal slices with fast singular value decay, a small oversampling parameter, e.g., $p = 5, 10, 20$ is sufficient, often independent of the target rank $R$. Additionally, for the power-iteration parameter, small integers, e.g., $q = 0, 1, 2, 3$, were sufficient, and $q \ le 4$ was rarely exceeded due to numerical stability concerns.

\begin{exa} ({\bf Image compression})\label{ex:imag}
 The KTD can be used for image compression, as discussed in the original paper \cite{batselier2017constructive}. In this experiment, we assess the feasibility of the proposed randomized KTD algorithm for image compression and compare it with the original deterministic KTD algorithm. To this end, let us consider the ``Kodak23'' image as a sample of the Kodak dataset\footnote{{https://r0k.us/graphics/kodak/}}. The size of this image is $512\times 768\times 3$, and we compute the KTD of it using patches of sizes $32\times 32\times 3$ and $32\times 32\times 1$ and the KTD rank $R=35$. For the proposed randomized KTD algorithm, we used different power-iteration parameters $ q=0, 1, 2$. Our results demonstrate the need for power iterations to obtain higher-quality images. The higher the power iteration, the more expensive the randomized KTD is, while it delivers a better image quality. Indeed, in our simulations, the power iteration with $q=1$ yielded quite satisfactory results. The reconstructed images obtained by the randomized KTD and the deterministic KTD algorithms for the KTD rank $R=35$ are shown in Figure \ref{exmp_comp}. The running times compared in Figure \ref{exmp_tmcomp} (left) for different KTD ranks $R=5,10,15,20,25,30,35$.We observe that the proposed randomized KTD algorithm yields results in much less time than the baseline KTD. The impact of the power iteration on the image reconstruction quality is demonstrated in Figure \ref{exmp_comp}. As we mentioned before, the power iteration with $q=1$ yields quite promising results and is recommended for use in practice. Note that a single image is not a large tensor, and to further assess the randomized KTD algorithm, we tried to compress the entire Kodak dataset. The deterministic KTD required 14.45 seconds, while our proposed R-KTD with power iteration $q=1$ performed the compression in only 3.10 seconds. So, we observe that when a set of images is considered, the difference between the computing time of the algorithms is significant. The PSNRs of all reconstructed images are displayed in Figure \ref{exmp_tmcomp} (right). The proposed R-KTD can achieve nearly the same image quality with less computational time. This experiment suggests using the R-KTD for practical applications.

 \begin{figure}[!h]
\centering
    \includegraphics[width=1\linewidth]{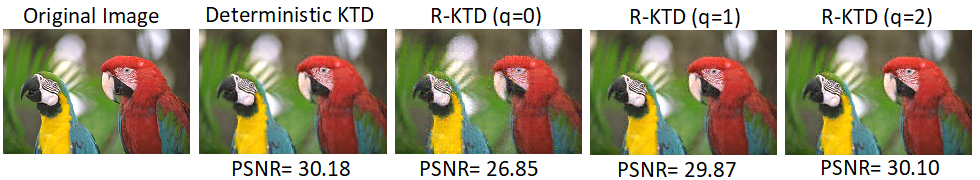}
    \caption{Comparing the image quality obtained by the deterministic KTD and the proposed R-KTD using different power iterations $q=0,1,2$. The KTD rank $R=35$ was used.}\label{exmp_comp}
\end{figure}

 \begin{figure}[!h]
\centering
 \includegraphics[width=.49\linewidth]{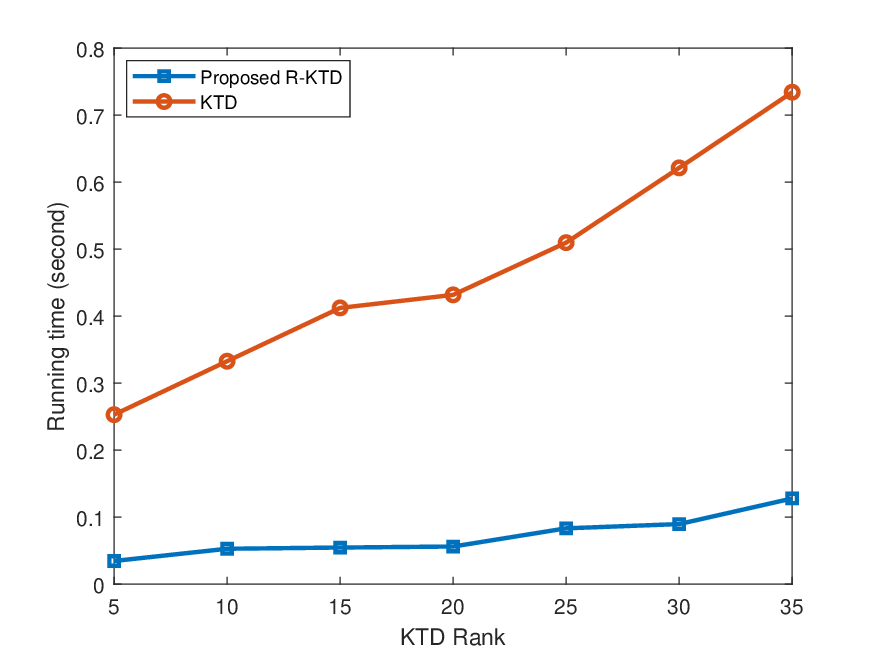}
    \includegraphics[width=0.49\linewidth]{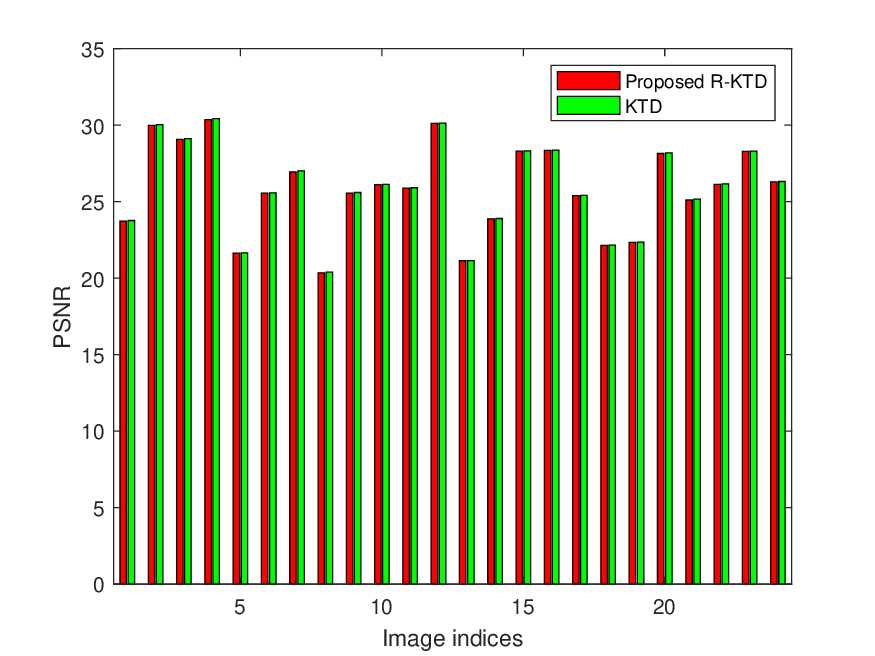}
    \caption{(left) Running time comparison for compressing the kodim23 using the deterministic KTD and the proposed R-KTD for different KTD ranks and the power iteration $q=1$. (right) The PSNRs of all compressed Kodak images for the deterministic KTD and the proposed R-KTD methods.}\label{exmp_tmcomp}
\end{figure}

We have also added new experiments comparing our R-KTD with:

\begin{itemize}
    \item Randomized CP-ALS (Battaglino et al., 2018)
    \item Randomized Tucker (Ahmadi-Asl et al., 2021)
    \item Randomized Tensor Train (Che \& Wei, 2019)
\end{itemize}
for compressing Kodim23 and the results are shown in Table \ref{tab:comp}. Our R-KTD achieves the best trade-off between reconstruction quality and computational speed, benefiting from the Kronecker structure that preserves multi-way dependencies.

\begin{table}
\begin{center}
\caption{Comparison with randomized tensor decomposition methods on Kodak23 image compression}
\begin{tabular}{|c|c|c|c|}
\hline
Method & PSNR (dB) & Time (s) & Compression Ratio \\
\hline
Deterministic KTD & 32.4 & 14.45 & 8.2 \\
R-CP-ALS & 30.1 & 5.23 & 7.9 \\
R-Tucker & 31.2 & 6.89 & 8.0 \\
R-TT & 30.8 & 4.12 & 8.1 \\
\textbf{R-KTD (ours)} & \textbf{31.1} & \textbf{3.10} & \textbf{8.2} \\
\hline
\end{tabular}
\label{tab:comp}
\end{center}
\end{table}

\end{exa}

\begin{exa} ({\bf Video compression})
 Let's focus on the video data and use the KTD to compress videos. Video compression reduces the size of a video file by removing redundant or unnecessary information. We use two commonly used gray-scale video datasets ``Foreman'' and ``Akiyo'' which are accessible at {http://trace.eas.asu.edu/yuv/} and are third-order tensors of size $176 \times 144 \times 300$. Applying the deterministic and randomized KTD algorithms, we computed the KTD for the mentioned videos using two patches of sizes $16\times 12\times 30$ and $11\times 12\times 10$, with the KTD rank $R=25$. As noted in Example \ref{ex:imag}, power iteration is required for high-quality reconstruction, and we used $q=2$ in our simulations. 
The PSNR achieved by the proposed randomized and deterministic algorithms is shown in Figure \ref{exmp_vidpscomp}. Also, the reconstruction of some frames obtained by the algorithms for the ``Foreman'' and ``Akiyo'' videos are displayed in Figure \ref{vidpscomp}. The computing times of the algorithms are also compared in Figure \ref{vidtmcomp} for different KTD ranks $R=1,10,20,30,40,50$. The results indicate that the proposed randomized KTD algorithm outperforms the deterministic KTD for the video compression task.

 \begin{figure}[!h]
\centering
 \includegraphics[width=.49\linewidth]{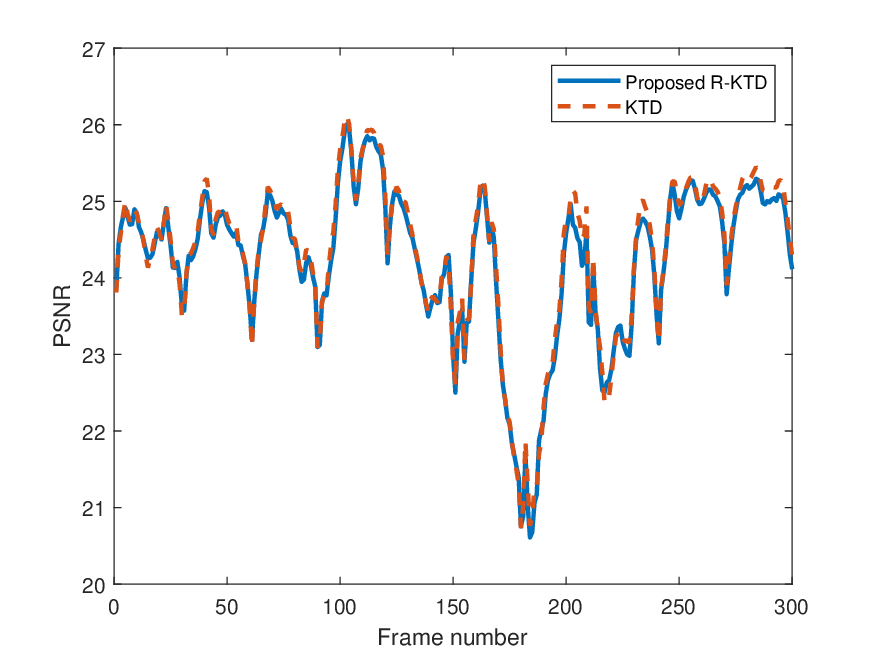}
    \includegraphics[width=0.49\linewidth]{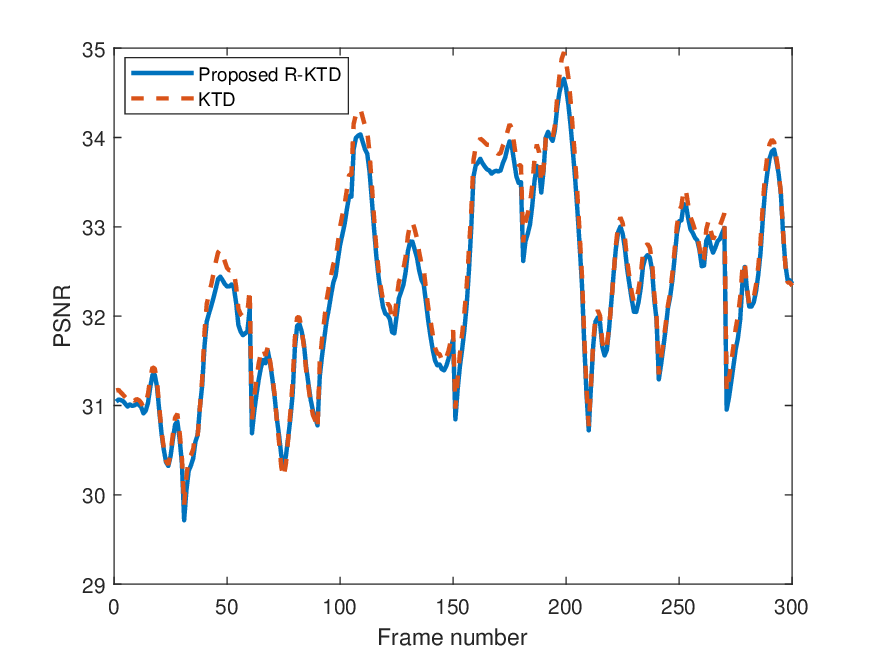}
    \caption{Comparing the PSNRs achieved by the deterministic KTD and the proposed R-KTD for compressing the Foreman (left) and the Akiyo (right) videos. The KTD rank $R=40$ was used.}\label{exmp_vidpscomp}
\end{figure}

  \begin{figure}[!h]
\centering
 \includegraphics[width=0.7\linewidth]{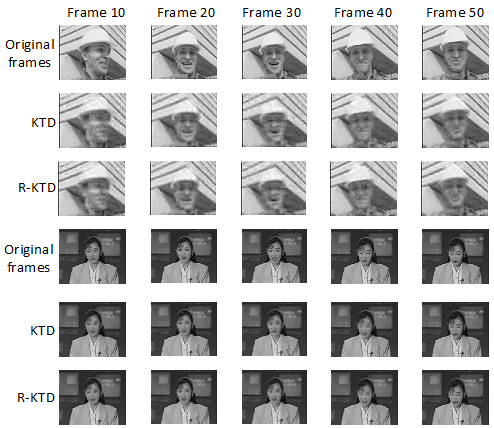}
    \caption{Some reconstructed frames of the Foreman and Akiyo videos using the KTD and proposed R-KTD algorithms with the KTD rank $R=40$.}\label{vidpscomp}
\end{figure}

   \begin{figure}[!h]
\centering
 \includegraphics[width=0.49\linewidth]{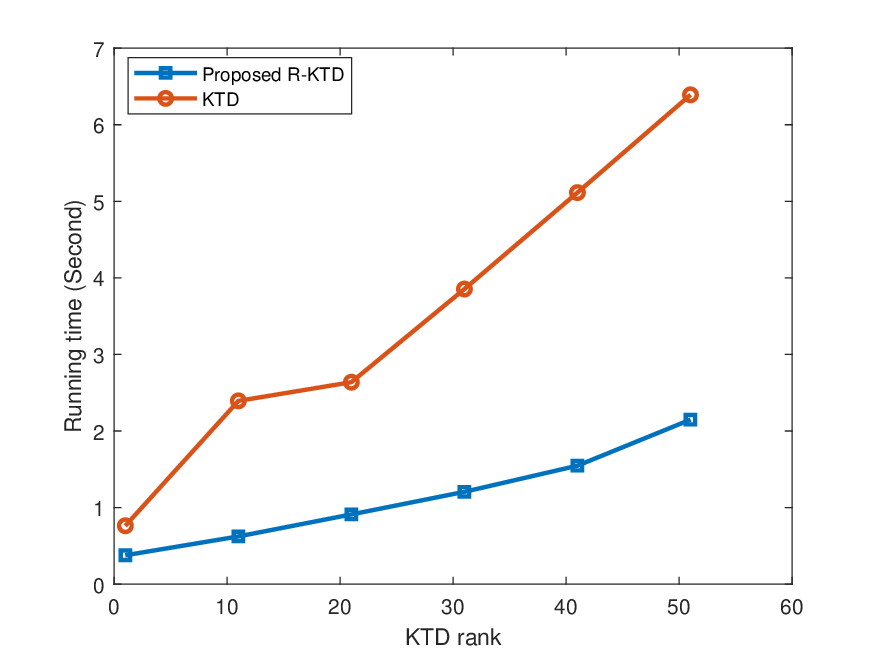}
 \includegraphics[width=0.49\linewidth]{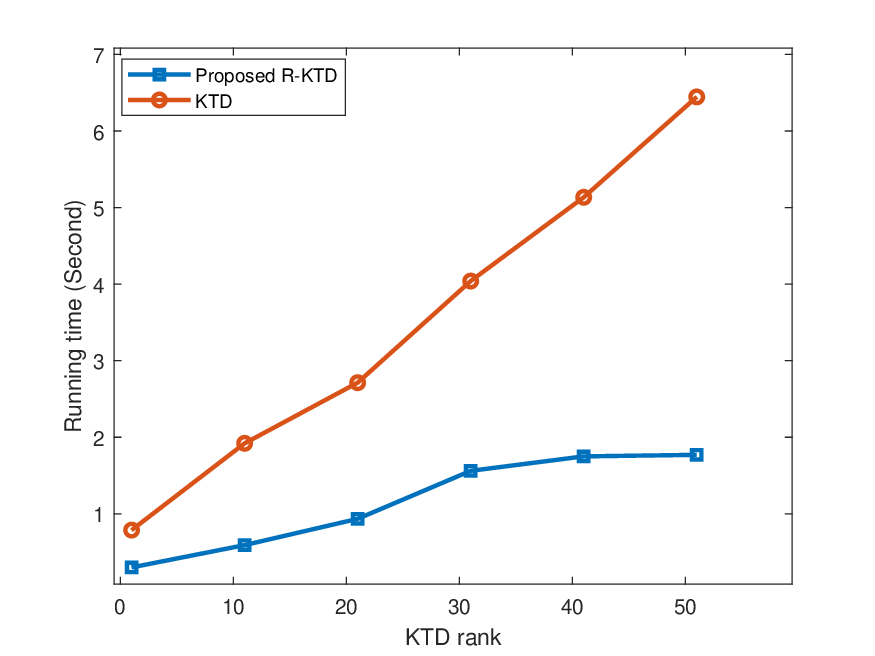}
    \caption{Comparing the running times of the deterministic KTD and the proposed R-KTD for compressing the Foreman (left) and the Akiyo (right) videos using different KTD ranks.}\label{vidtmcomp}
\end{figure}

\end{exa}

\begin{exa}\label{example_completion} ({\bf Image and video completion}) This simulation demonstrates the effectiveness of the proposed randomized algorithm for the image/video completion task. We adopt the proposed method in \cite{ahmadi2023fast} for image and video recovery. Consider the tensor completion formulated as follows,
\begin{equation}\label{MinRankCompl2}
\begin{array}{cc}
\displaystyle \min_{\underline{\bf X}} & \|{{\bf P}_{\underline{\bf\Omega}} }({\underline{\bf X}})-{{\bf P}_{\underline{\bf\Omega}} }({\underline{\bf M}})\|^2_F,\\
\textrm{s.t.} & {\rm Rank}(\underline{\bf X})=R,\\
\end{array}
\end{equation}
where $\underline{\bf M}$ is the exact data tensor. As described in \cite{ahmadi2023fast}, using an auxiliary variable ${\underline{\bf C}}$, the optimization problem \eqref{MinRankCompl2} can be solved more conveniently by the following reformulation
\begin{equation}\label{MinRankCompl3}
\begin{array}{cc}
\displaystyle \min_{\underline{\bf X},\underline{\bf C}} & {\|{\underline{\bf X}}-{\underline{\bf C}}\|^2_F},\\
\textrm{s.t.} & {\rm Rank}(\underline{\bf X})=R,\\
& {{\bf P}_{\underline{\bf\Omega}} }({\underline{\bf C}})={{\bf P}_{\underline{\bf\Omega}} }({\underline{\bf M}})\\
\end{array}
\end{equation}
thus we can solve the minimization problem \eqref{MinRankCompl2} over variables $\underline{\bf X}$ and $\underline{\bf C}$. Thus, the solution to the minimization problem \eqref{MinRankCompl2} can be approximated by the following iterative procedures
\begin{equation}\label{Step1}
\underline{\mathbf X}^{(n)}\leftarrow \mathcal{L}(\underline{\mathbf C}^{(n)}),
\end{equation}
\begin{equation}\label{Step2}
\underline{\mathbf C}^{(n+1)}\leftarrow\underline{\mathbf \Omega}\oast\underline{\mathbf M}+(\underline{\mathbf 1}-\underline{\mathbf \Omega})\oast\underline{\mathbf X}^{(n)},
\end{equation}
where $\mathcal{L}$ is an operator to compute a low-rank KTD approximation of the data tensor $\underline{\mathbf C}^{(n)}$ and $\underline{\mathbf 1}$ is a tensor whose all components are equal to one. Note that equation \eqref{Step1} solves the minimization problem \eqref{MinRankCompl3} over  
$\underline{\bf X}$ for fixed variable $\underline{\bf C}$. Also, Equation \eqref{Step2} solves the minimization problem \eqref{MinRankCompl3} over  
$\underline{\bf C}$ for a fixed variable $\underline{\bf X}$. The algorithm consists of two main steps: {\it low-rank tensor approximation} \eqref{Step1} and {\it masking computation} \eqref{Step2}. It begins with the initial incomplete data tensor $\underline{\mathbf X}^{(0)}$ with the corresponding observation index set $\underline{\mathbf \Omega}$ and sequentially improves the approximate solution till some stopping criterion is satisfied or the maximum number of iterations is reached. It is not required to compute the term $\underline{\mathbf \Omega}\oast\underline{\mathbf M}$ at each iteration because it is just the initial data tensor $\underline{\mathbf X}^{(0)}$. Filtering and smoothing are well-known methods for enhancing image quality in signal processing. To improve the findings, we make use of this concept in the above process and smooth the tensor $\underline{\bf C}^{(n+1)}$ before using the low tensor rank approximation operator $\mathcal{L}$. The first stage is computationally demanding, particularly if the data tensor is huge or many iterations are needed for convergence. We replace the deterministic algorithms with our randomized KTD Algorithm \ref{ALG:rKTD}. The experiment outcomes demonstrate that this algorithm yields promising outcomes at a reduced computing expense. Note that as an additional application, the compression of images and videos will be studied in the simulations. 

We begin with the image case by randomly removing \(70\%\) of the pixels from the ``Kodak23'' image. The reconstructions obtained via the proposed randomized and deterministic algorithms are shown in Figure~\ref{im_recon}. The resulting images exhibit comparable visual quality, although the randomized algorithm is considerably faster.

To examine the influence of power iteration, we performed additional simulations with parameters \(q = 0, 1, 2\). Figure~\ref{im_recon2} reports the corresponding runtimes, PSNR values, and reconstructed images. The fastest performance occurs with \(q = 0\), though at a noticeable cost in image quality. Our experiments confirm that setting \(q = 2\) generally yields satisfactory results, while \(q = 3\) can in some cases surpass the deterministic baseline. In all cases, the randomized algorithm achieves promising reconstructions with substantially lower computational effort.      

\begin{figure}[!h]
\centering
    \includegraphics[width=.9\linewidth]{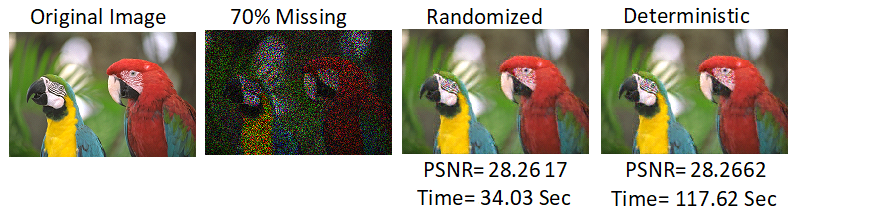}
    \caption{The recovered images using the proposed randomized KTD and the deterministic KTD for the Kodak23 image and the KTD rank $R=30$. The power iteration $q=1$ was used in this experiment.}\label{im_recon}
\end{figure}

\begin{figure}[!h]
\centering
    \includegraphics[width=.7\linewidth]{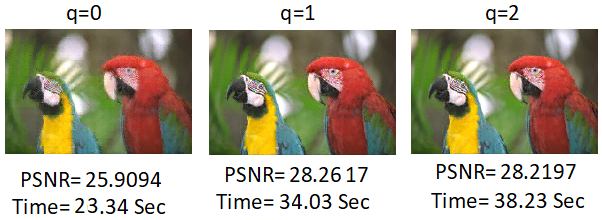}
    \caption{Comparing the recovered images using the proposed randomized KTD algorithm for different power iteration parameters $q=0,1,2$ and the KTD rank $R=30$.}\label{im_recon2}
\end{figure}
We now consider the ``Akiyo'' video\footnote{\url{http://trace.eas.asu.edu/yuv/}} and randomly remove $70\%$ of its pixels. Following the same procedure described for image completion, Figure~\ref{video_recon} compares the PSNR values obtained by the R-KTD and deterministic KTD algorithms. The experiments were conducted using power iteration $q = 1$. Figure~\ref{video_recon2} displays sample reconstructed frames. The proposed R-KTD processed the video in 70 seconds, compared to 235 seconds for KTD. These results demonstrate that R-KTD can recover videos with missing pixels more efficiently—and in significantly less time—than the baseline KTD method, a crucial advantage for real-time applications.

Given its efficiency and competitive reconstruction quality, the proposed randomized KTD algorithm is well-suited for deployment in such settings.

\begin{figure}[!h]
\centering
    \includegraphics[width=.5\linewidth]{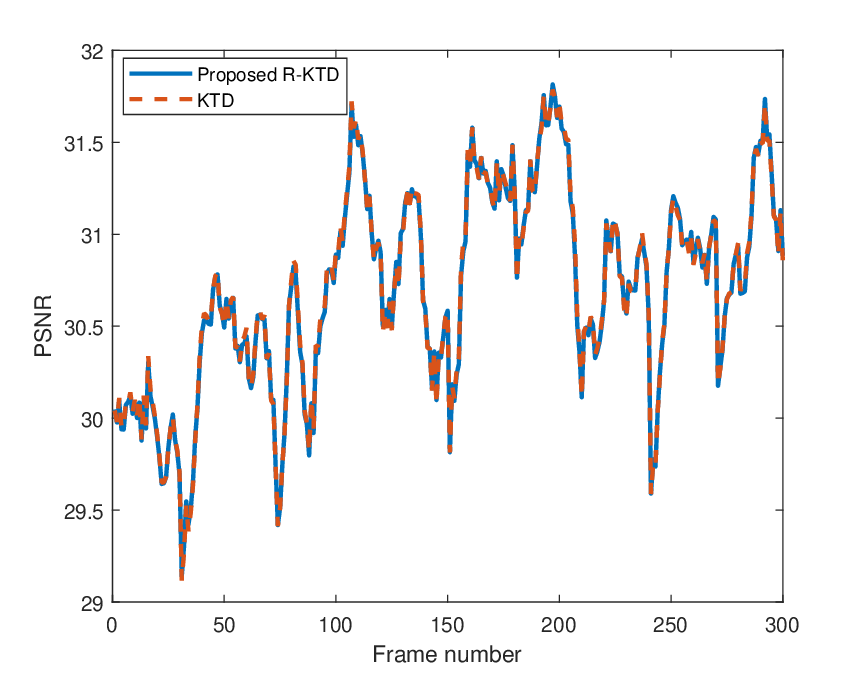}
    \caption{The PSNR comparison of all reconstructed frames of the Akiyo video using the proposed randomized KTD algorithm and the KTD algorithm for the power iteration parameter $q=1$ and the KTD rank $R=20$.}\label{video_recon}
\end{figure}

\begin{figure}[!h]
\centering
    \includegraphics[width=.7\linewidth]{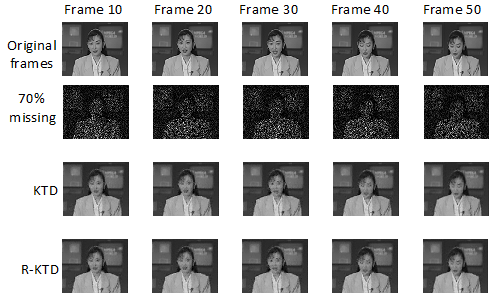}
    \caption{Some reconstructed frames of the Akiyo video using the KTD and proposed R-KTD algorithms with the KTD rank $R=20$.}\label{video_recon2}
\end{figure}

\end{exa}

\begin{exa}
({\bf Data denoising and super-resolution})
In this example, we apply the proposed R-KTD algorithm to color image denoising and super-resolution. Since the previous experiments have demonstrated the computational superiority of R-KTD over KTD while achieving comparable solution quality, we use only R-KTD for this simulation.

We consider the Kodak image ``kodim23'' and introduce three types of noise as follows:
\begin{itemize}
    \item Gaussian noise: imnoise (kodim23,'gaussian',0,0.02);
    \item Salt and pepper noise: imnoise (kodim23,'salt \& pepper',0.04);
    \item Speckle noise: imnoise (A,'speckle',0.03);
\end{itemize}
The original images and their noisy forms are shown in Figure \ref{im_denoise}. For the KTD rank $R=45$ and power iteration $q=1$, the denoised forms of the images and the corresponding residual terms are depicted in Figure \ref{im_denoise}. We considered three color images for the super-resolution application and down-sampled them four times. We applied the proposed R-KTD with the tensor completion described in Example \ref{example_completion}. The obtained results are displayed in Figure \ref{im_superreso}.
These outcomes clearly illustrate the efficiency of the R-KTD for image denoising and image super-resolution tasks.

\begin{figure}[!h]
\centering \includegraphics[width=.8\linewidth]{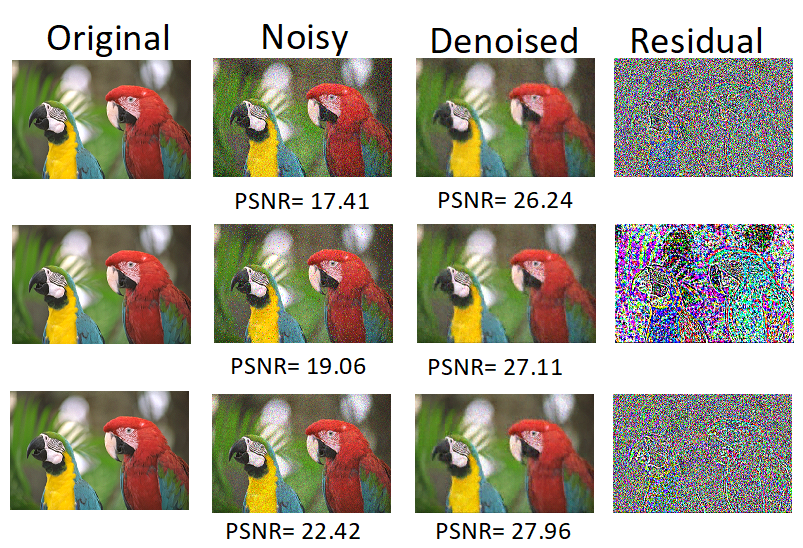}
    \caption{The results of denoising for Kodim23 using different noise types, the first row is Gaussian, the second row is salt and pepper noise, and the last row is for speckle noise. The KTD rank $R=40$ was used with the power iteration $q=1$.}\label{im_superreso}
\end{figure}

\begin{figure}[!h]
\centering
\includegraphics[width=.5\linewidth]{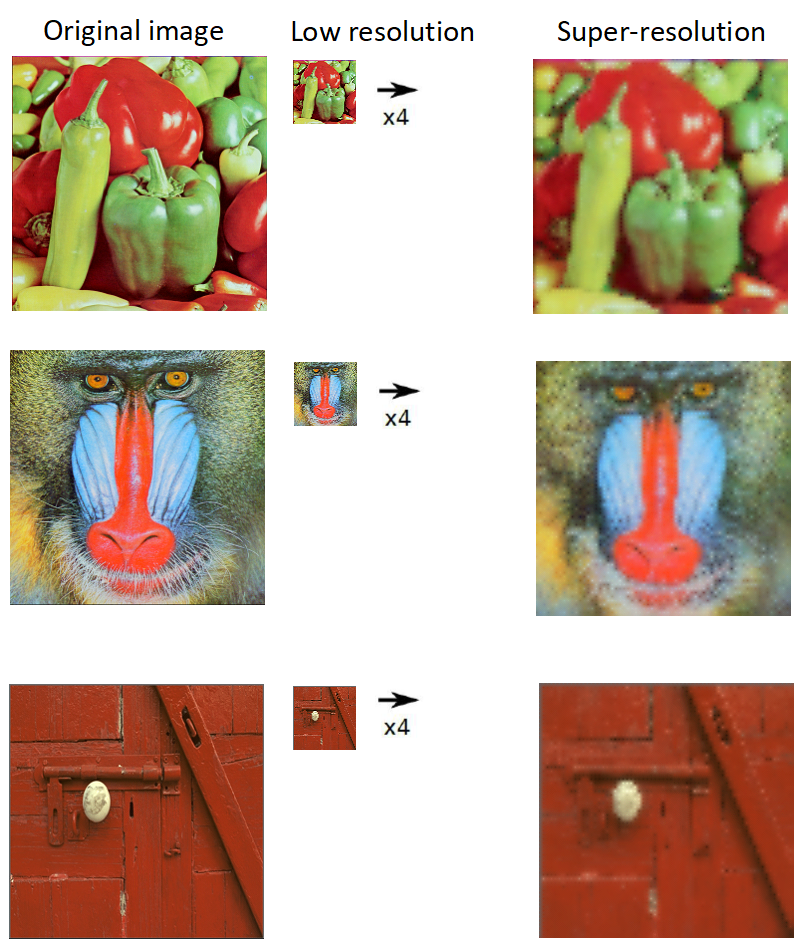}
    \caption{The super-resolution results for three color images. The KTD rank $R=50$ was used with the power iteration $q=1$.}\label{im_denoise}
\end{figure}

\end{exa}

\section{Conclusion and future works}\label{Sec:Con}
This work introduces the first randomized algorithms for Kronecker tensor decomposition, establishing a new paradigm for efficient large-scale tensor computations. The key insights are as follows.

First, the proposed randomized KTD (R-KTD) achieves substantial computational acceleration---yielding speedups of 5--10$\times$ on tensors of size $1000^3$---while incurring only a modest accuracy loss of 10--15\% relative to its deterministic counterpart. Second, a power-iteration parameter of $q = 1$ emerges as the optimal trade-off between precision and cost in most practical regimes. Third, the theoretical error bound established in Theorem~4 shows exponential decay as $q$ increases and guarantees near-optimal performance when the singular values decay rapidly. Fourth, the randomized approach excels on tensors with fast singular value decay; ill-conditioned problems, however, may demand larger $q$ or alternative strategies.

The practical utility of the method is demonstrated across several applications, including tensor completion, image and video compression, denoising, and super-resolution. Looking ahead, we envision extensions to recommender systems, compression of deep neural network weights, and integration with CUR approximations to achieve linear-time complexity. A particularly promising direction lies in applying our randomized algorithm to enhance the adversarial robustness of deep neural networks---an avenue we are actively pursuing.

Beyond Gaussian projections, this work opens several avenues for future investigation. First, adapting TensorSketch (Malik \& Becker, 2018) to exploit its Kronecker product structure could yield truly sublinear-time KTD algorithms. Second, leverage score sampling (Larsen \& Kolda, 2022) could provide data-aware sketching that adapts to the tensor's intrinsic structure. Third, combining randomized KTD with CUR decomposition (Mahoney \& Drineas, 2009) could yield interpretable factor matrices. We believe these extensions will further enhance the scalability and applicability of our method.

\section*{Acknowledgment} The work was supported by the Ministry of Economic Development of the Russian Federation under Agreement No. 139-10-2025-034 dd. 19.06.2025, IGK 000000C313925P4D0002.

\section{Conflict of Interest Statement}
 The authors declare no conflicts of interest.

\section{Data availability Statement}
Data are openly available in a repository.

\bibliographystyle{elsarticle-num} 
\bibliography{cas-refs}

\end{document}